\documentclass[12pt]{amsart}
\usepackage{amssymb}
\usepackage{amscd}
\usepackage{amsfonts}
\usepackage{graphicx}
\usepackage{latexsym}
\usepackage[all]{xy}
\usepackage{verbatim}
\usepackage{color}
\usepackage[margin=1in]{geometry}
\newtheorem{theorem}{Theorem}[section]
\newtheorem{lemma}[theorem]{Lemma}
\newtheorem{proposition}[theorem]{Proposition}
\newtheorem{corollary}[theorem]{Corollary}
\theoremstyle{definition}
\newtheorem{definition}[theorem]{Definition}
\newtheorem{example}[theorem]{Example}
\newtheorem{conjecture}[theorem]{Conjecture}

\newtheorem{remark}[theorem]{Remark}

\newcommand{\id}{\text{id}}

\renewcommand{\Vec}{\operatorname{\operatorname{\mathsf{Vec}}}}

\DeclareMathOperator{\Aut}{\operatorname{\mathsf{Aut}}}
\DeclareMathOperator{\Ext}{\operatorname{\mathsf{Ext}}}

\DeclareMathOperator{\Rep}{\operatorname{\mathsf{Rep}}}
\DeclareMathOperator{\Corep}{\operatorname{\mathsf{Corep}}}

\DeclareMathOperator{\ind}{\operatorname{\mathsf{ind}}}
\DeclareMathOperator{\Hom}{\operatorname{\mathsf{Hom}}}
\DeclareMathOperator{\Alt}{\operatorname{\mathsf{Alt}}}
\DeclareMathOperator{\Sym}{\operatorname{\mathsf{Sym}}}
\DeclareMathOperator{\Quad}{\operatorname{\mathsf{Quad}}}

\newcommand{\rev}{\text{rev}}

\newcommand{\ev}{\text{ev}}

\newcommand{\gr}{\text{gr}}

\newcommand{\eps}{\varepsilon}

\newcommand{\C}{\mathcal{C}}

\newcommand{\Z}{\mathcal{Z}}
\newcommand{\YD}{{}_{\Gamma} ^{\Gamma} \mathcal{YD}}

\newcommand{\A}{\mathcal{A}}

\renewcommand{\1}{_{(1)}}
\renewcommand{\2}{_{(2)}}
\newcommand{\3}{_{(3)}}

\newcommand{\ot}{\otimes}

\newcommand{\beq}{\begin{equation}}
\newcommand{\eeq}{\end{equation}}

\newcommand{\bpf}{\begin{proof}}
\newcommand{\epf}{\end{proof}}

\newcommand{\bth}{\begin{theorem}}
\renewcommand{\eth}{\end{theorem}}
\newcommand{\bpr}{\begin{proposition}}
\newcommand{\epr}{\end{proposition}}
\newcommand{\ble}{\begin{lemma}}
\newcommand{\ele}{\end{lemma}}
\newcommand{\bco}{\begin{corollary}}
\newcommand{\eco}{\end{corollary}}
\newcommand{\bde}{\begin{definition}}
\newcommand{\ede}{\end{definition}}
\newcommand{\bex}{\begin{example}}
\newcommand{\eex}{\end{example}}
\newcommand{\bre}{\begin{remark}}
\newcommand{\ere}{\end{remark}}
\newcommand{\bcj}{\begin{conjecture}}
\newcommand{\ecj}{\end{conjecture}}

\newcommand{\End}{\text{End}}%{\mathcal{E}\it{nd}}

\newcommand{\Binv}{\text{B}_\text{inv}^2}
\newcommand{\Hinv}{\text{H}_\text{inv}^2}
\newcommand{\Zinv}{\text{Z}_\text{inv}^2}

\begin{document}

%\title[Co-quasitriangular pointed Hopf algebras]{Co-quasitriangular pointed Hopf algebras}
\title[Pointed braided tensor categories]{Pointed braided tensor categories}
\author{Costel-Gabriel Bontea}
\address{Department of Mathematics and Statistics,
University of New Hampshire,  Durham, NH 03824, USA}
\email{costel.bontea@gmail.com}
\author{Dmitri Nikshych}
\address{Department of Mathematics and Statistics,
University of New Hampshire,  Durham, NH 03824, USA}
\email{dmitri.nikshych@unh.edu}

\begin{abstract}
We classify finite pointed braided  tensor categories admitting a fiber functor 
in terms of  bilinear forms on symmetric Yetter-Drinfeld modules over abelian groups.  We describe
the groupoid formed by braided equivalences of such categories in terms of certain metric data, 
generalizing the well-known result of Joyal and Street \cite{JS93} for  fusion categories.
We study symmetric centers and  ribbon structures of pointed braided  tensor categories   
and examine their Drinfeld centers.
\end{abstract}

\maketitle

\baselineskip=18pt
%%%%%%%%%%%%%%%%%%%%%%%%%%%%%%%%%%%%%%%%%%%%%%%%%%%%%%

%%%%%%%%%%%%%%%%%%%%%%%%%%%%%%%%%%%%%%%%%%%%%%%%%%%%%%%%%%%%%%%%%%%%%%%%%%%
%%%%%%%%%%%%%%%%%%%%%%%%%%%%  PRELIMINARIES   %%%%%%%%%%%%%%%%%%%%%%%%%%%%%%%%%%%%%
%%%%%%%%%%%%%%%%%%%%%%%%%%%%%%%%%%%%%%%%%%%%%%%%%%%%%%%%%%%%%%%%%%%%%%%%%%%

\section{Introduction}

In this paper, we work over an algebraically closed field $k$ of characteristic $0$.
All tensor categories are assumed to be $k$-linear and finite. All Hopf algebras and modules over them 
are defined over $k$ and are assumed to be finite dimensional.  
%We refer the reader to
%\cite{M93} and \cite{EGNO15} for basics of Hopf algebras and tensor categories. 

A tensor category is called {\em pointed} if all its simple objects are invertible. 
An example of such a category is the category of  finite dimensional corepresentations $\Corep(H)$ 
of a  pointed Hopf algebra $H$.  Any pointed tensor category  admitting a fiber
functor is equivalent to some $\Corep(H)$. The classification of  pointed Hopf algebras
having  abelian group of group-like elements is nearing its completion, see \cite{A14}. 

This paper deals with classification of {\em braided} pointed  tensor categories.  Such a classification
is well known in the semisimple case, i.e., for fusion categories. It was proved by Joyal and Street \cite{JS93}
that the $1$-categorical truncation of the $2$-category of braided fusion categories is equivalent to the category of pre-metric groups, i.e.,
pairs $(\Gamma,\, q)$, where $\Gamma$ is a finite abelian group and $q:\Gamma\to k^\times$ is a quadratic form. 
In particular, braidings on pointed fusion categories are in bijection with abelian $3$-cocycles.  In the presence
of a fiber functor such cocycles are precisely bicharacters on abelian groups. Explicitly, a braided fusion category
having a fiber functor is equivalent to $\C(\Gamma,\, r_0) := \Corep(k[\Gamma],\, r_0)$, where 
the $r$-form $r_0$ is given by  a bicharacter on $\Gamma$.

In this work we extend the above results to non-semisimple braided tensor categories.  We classify
co-quasitriangular pointed Hopf algebras up to tensor equivalence of their corepresentation categories,
thereby obtaining classification of braided tensor categories having a fiber functor.

\begin{theorem}
\label{Thm 1}
Let $\C$ be a pointed braided tensor category having a fiber functor. 
Then $\C$ is completely determined by  a finite abelian group $\Gamma$,
 a bicharacter $r_0:\Gamma\times \Gamma \to k^\times$, an object   $V\in \Z_{sym}(\C(\Gamma,\, r_0))_{-}$,
and a morphism $r_1: V \ot V\to k$.  More precisely, 
\begin{equation}
\label{C as Corep}
\C \cong \C(\Gamma,\, r_0,\, V,\, r_1) := \Corep(\mathfrak{B}(V) \# k[\Gamma],\, r), 
\end{equation}
where $r|_{\Gamma\times\Gamma}= r_0$ and $r|_{V\ot V} =r_1$. 
\end{theorem}

Here the symmetric center $\Z_{sym}(\C(\Gamma,\, r_0))$ has a canonical (possibly trivial) grading by
$\mathbb{Z}/2\mathbb{Z}$: 
\[
\Z_{sym}(\C(\Gamma,\, r_0))  = \Z_{sym}(\C(\Gamma,\, r_0))_{+} \oplus \Z_{sym}(\C(\Gamma,\, r_0))_{-},
\]
where $\Z_{sym}(\C(\Gamma,\, r_0))_{+}$ denotes the maximal Tannakian subcategory.  
% We call an object $V$ satisfying the hypothesis of Theorem~\ref{Thm 1} a quantum linear space of symmetric type, cf.\ \cite{AS98}.

Theorem~\ref{Thm 1} is proved in Section~\ref{Class of CQT}, where details of the construction of $r$ can be found. 
Quasitriangular structures  on $\mathfrak{B}(V) \# k[\Gamma]$ were explicitly described by Nenciu in \cite{Ne04} 
in terms of  generators of $\Gamma$ and a basis of $V$.  The classification of co-quasitriangular structures
can be obtained  by duality.  Theorem~\ref{Thm 1} says that {\em every} pointed co-quasitriangular Hopf algebra
is equivalent to the above by a $2$-cocycle deformation.  Also, our description of $r$-forms is given in invariant terms 
and avoids the use of bases and generators. 

We describe the symmetric center of $\C(\Gamma,\, r_0,\, V,\, r_1)$ and show that 
a pointed braided tensor category is not factorizable
unless it is semisimple.  We also show that this category is always ribbon and classify its ribbon  structures. 

We obtain a parameterization of pointed braided tensor categories similar to the parameterization of braided fusion
categories by quadratic forms \cite{JS93}. Namely, we introduce the groupoid of {\em metric quadruples}
$(\Gamma,\, q,\,V,\, r)$, where $\Gamma$ is a finite abelian group, $q: \Gamma\to k^\times$ is a diagonalizable quadratic form,
$V$ is an object in $\Z_{sym}(  \C(\Gamma,\, q))_{-}$, and $r:V\ot V \to k$ is an alternating  morphism.

\begin{theorem}
\label{Thm 2}
The groupoid of isomorphism classes of equivalences of pointed braided tensor categories having a fiber functor is equivalent
to the groupoid of metric quadruples.
\end{theorem}

Theorem~\ref{Thm 2} is proved in Section~\ref{Metric 4sect}. The braiding symmetric form $r$ can be canonically 
recovered from the restriction of the squared braiding on  two-dimensional objects of a category, 
see Remark~\ref{r from Ext}. 

The Drinfeld center of $\C(\Gamma,\, r_0,\, V,\, r_1)$ is not pointed when  $V \neq 0$.  We show that when $V \cong V^*$ 
the trivial component of the universal grading of $\Z(\C(\Gamma,\, r_0,\, V,\, r_1))$ is pointed  
and corresponds to a certain braided vector space with 
a symplectic bilinear form (the Drinfeld double of $V$),
see Section~\ref{double of quantum linear space} and  Theorem~\ref{Zad} for details. 

The paper is organized as follows.  

Section~\ref{sect: Prelim} contains background material about braided tensor categories,
co-quasitriangular Hopf algebras and their twisting deformations by $2$-cocycles.

In Section~\ref{Sect:QLS} we discuss quantum linear spaces of symmetric type and their bosonizations
$\mathfrak{B}(V)\#k[\Gamma]$.
We use Mombelli's classification of Galois objects  for quantum linear spaces \cite{Mo11} to obtain a classification
of $2$-cocycles on such bosonozations in terms of $V$ (equivalently, we obtain a classification of fiber functors on corresponding 
corepresentation categories), see Proposition~\ref{all twists on H}.  We also compute the second invariant
cohomology group of $\mathfrak{B}(V)\#k[\Gamma]$ in Proposition~\ref{Hinv computed}.

In Section~\ref{Class of CQT} we prove Theorem~\ref{Thm 1}.

In Section~\ref{Sect:Zsym} we study the symmetric center of $\C(\Gamma,\, r_0,\, V,\, r_1)$. It turns
out that this center is always non-trivial if $V\neq 0$. 

In Section~\ref{Sect:ribbon} we show that  a pointed braided tensor  category always has a ribbon
structure and classify all such structures, improving the result of \cite{Ne04}. 

In Section~\ref{Metric 4sect} we prove Theorem~\ref{Thm 2}. The correspondence between pointed braided tensor categories
and metric quadruples is useful, in particular, for computing the groups of autoequivaleces, 
see Corollary~\ref{group of autoequivalences}.

Finally, Section~\ref{Sect: Drinfeld center} contains a description of the Drinfeld center of the  pointed braided
tensor category $\C(\Gamma,\, r_0,\, V,\, r_1)$ when $V$ is self-dual. This category is no longer pointed
if $V \neq 0$, but it has a faithful grading with a pointed trivial component. We describe the structure of this component
and show that it is the corepresentation category of  the bosonization of the Drinfeld double of $V$.

\textbf{Acknowledgments.}
We are grateful to Adriana Nenciu  for helpful discussions.
The work of the second author  was partially supported  by the NSA grant H98230-16-1-0008.

%%%%%%%%%%%%%%%%%%%%%%%%%%%%%%%%%%%%%%%%%%%%%%%%%%%%%%%%%%%%%%%%%%%%%%%%%%%
%%%%%%%%%%%%%%%%%%%%%%%%%%%%  PRELIMINARIES   %%%%%%%%%%%%%%%%%%%%%%%%%%%%%%%%%%%%%
%%%%%%%%%%%%%%%%%%%%%%%%%%%%%%%%%%%%%%%%%%%%%%%%%%%%%%%%%%%%%%%%%%%%%%%%%%%

\section{Preliminaries}
\label{sect: Prelim}

\subsection{Finite tensor categories and Hopf algebras}

We assume familiarity with basic results of the theory of finite
tensor categories   \cite{EGNO15},
and the theory of Hopf algebras \cite{M93, R11}. 
%All tensor categories
%in this paper are assumed to be finite and all Hopf algebras and
%(co-)modules over them  are assumed  to be finite dimensional.

%In this paper $k$ will denote an algebraically closed field of
%characteristic zero. All vector spaces, Hopf algebras, tensor
%products, etc., are over $k$ and we will write $\ot$ instead of
%$\ot_{k}$. 

For a Hopf algebra $H$ we denote by $\Delta$,
$\varepsilon$, $S$ the comultiplication, counit, and
antipode of $H$, respectively. We make use of Sweedler's summation
notation: $\Delta (x) = x_{(1)} \ot x_{(2)}$, $x \in H$. 
We denote by $*$ the convolution, i.e., the multiplication in the dual Hopf algebra. 
By $H^{\text{cop}}$ we denote the co-opposite Hopf algebra of $H$,
i.e., the algebra $H$ with co-multiplication $\Delta^{\text{op}}
(x) = x_{(2)} \ot x_{(1)}$, $x\in H$.
We denote $G(H)$ the group of group-like
elements of $H$ (i.e., elements $g\in H$ such that $\Delta(g)=g\ot g$).
Also,  we denote $H^+ = \mbox{Ker}(\varepsilon)$.

%Let $C$ be a coalgebra and let $A$ be an algebra. There is an algebra structure
%on $\Hom(C,\,A)$ with the multiplication given by the {\em convolution}
%\[
%f*g(x)= f(x\1) g(x\2),\qquad f,\,g\in \Hom(C,\,A)
%\]
%and the identity $x\mapsto \eps(x)1,\, x\in C$. 

Let $\Rep (H)$ and $\Corep (H)$ be tensor categories of left
$H$-modules and right $H$-comodules, respectively, over a Hopf algebra $H$.  Note that
there is a canonical tensor equivalence between $\Corep(H)$ and
$\Rep (H^{*})$. In general, a tensor category $\mathcal{C}$ is
equivalent to the co-representation category of some Hopf algebra
$H$ if and only if there exists a fiber functor (i.e., an exact
faithful tensor functor) $F : \mathcal{C} \to \Vec$, where $\Vec$
is the tensor category of $k$-vector spaces.

A tensor category is \textit{pointed} if all of its simple objects
are invertible with respect to the tensor product. A Hopf algebra
$H$ is \textit{pointed} if $\Corep (H)$ is pointed. The
classification of finite dimensional pointed Hopf algebras is
still an open problem, though important progress has been made so
far (see \cite{A14} and the references therein). The best
understood class is that of pointed Hopf algebras with abelian
coradical \cite{AS10}. It was shown by  Angiono \cite{An13}
that such a Hopf algebra $H$ is generated by its group-like
elements and  skew-primitive elements.

%%%%%%%%%%%%%%%%%%%%%%%%%%%%%%%%%%%%%%%%%%%%%%%%%%%%%%%%%%%%%
\subsection{Braided tensor categories and co-quasitriangular Hopf algebras}

A \textit{braiding} on a finite tensor category $\mathcal{C}$ is a
natural isomorphism
\[
c_{X, Y} : X \ot Y \to Y \ot X,\qquad X,\, Y \in \mathcal{C},
\]
satisfying the hexagon axioms. A \textit{braided}
tensor category is a pair consisting of a tensor category and a
braiding on it.

A \textit{co-quasitriangular} Hopf algebra is a pair $(H, r)$,
where $H$ is a Hopf algebra and $r : H \ot H \to k$ is a
convolution invertible linear map, called an {\em $r$-form},
satisfying the following conditions:
\begin{eqnarray}
x_{(1)} y_{(1)} r (y_{(2)}, x_{(2)}) &=& r (y_{(1)}, x_{(1)})
y_{(2)}x_{(2)} \label{cqt1}, \\
r (x, yz) &=& r (x_{(1)}, z) r (x_{(2)}, y) \label{cqt2},\\
r (xy, z) &=& r (x, z_{(1)}) r (y, z_{(2)}) \label{cqt3},
\end{eqnarray}
for all $x,\,y,\,z\in H$.

\begin{remark}
\label{r-form =
morphism} 
Let $r : H \ot H \to k$ be a linear map and let $\varphi_{r}
: H \to H^{* \textnormal{cop}}$ be defined by  $\varphi_{r} (x) = r (x, -)$, for
all $x \in H$. Then $r$ satisfies \eqref{cqt2}, respectively
\eqref{cqt3}, if and only if $\varphi_{r}$ is a coalgebra map,
respectively an algebra map.
\end{remark}

There is a bijective  correspondence between the set of $r$-forms
on a Hopf algebra $H$ and the set of braidings on $\Corep (H)$.
The braiding corresponding to $r:H\ot H\to k$ is given by
\begin{equation}
\label{braiding from r-form}
c_{U, V} : U \ot V \to V \ot U, \quad u \ot v \mapsto
\sum r (u_{(1)}, v_{(1)}) v_{(0)} \ot u_{(0)},
\end{equation}
where $U,\,V$ are $H$-comodules and $u\in U,\, v\in V$.

We denote by $\Corep(H,\, r)$ the braided tensor category
$\Corep (H)$ with braiding given by~$r$.

%The dual notion to co-quasitriangular Hopf algebra is that of
%\textit{quasitriangular} Hopf algebra. This is a pair $(H, R)$,
%where $H$ is a Hopf algebra and $R \in H \ot H$ is an invertible
%element, called an \textit{$R$-matrix}, satisfying conditions dual
%to \eqref{cqt1}-\eqref{cqt3}.  There is a bijective  correspondence between the set
%of $R$-matrices on Hopf algebra $H$ and the set
%of braidings on $\Rep (H)$.
%The braiding corresponding to  $R$ is given by
%$$
%c_{U, V} : U \ot V \to V \ot U, \quad u \ot v \mapsto \tau \big( R
%(u \ot v) \big),
%$$
%where $U,\,V$ are $H$-modules, $u\in U,\, v\in V$, and $\tau: H\ot H\to H\ot H$
%is the flip map.
%
%We denote by $\Rep(H,\, R)$ the braided tensor category
%$\Rep (H)$ with braiding given by $R$.

%%%%%%%%%%%%%%%%%%%%%%%%%%%%%%%%%%%%%%%%%%%%%%%%%%%%%%%%%%%%%
\subsection{Ribbon categories and ribbon elements}
\label{prelim: ribbon}

A \textit{ribbon} tensor category is a braided tensor category
$\mathcal{C}$ together with a \textit{ribbon structure} on it,
i.e., an element $\theta \in \Aut(\id_{\mathcal{C}})$ such that
\begin{eqnarray}
\theta_{X \ot Y} &=& (\theta_{X} \ot \theta_{Y}) \circ c_{Y, X}
\circ c_{X, Y} \\
(\theta_{X})^{*} &=& \theta_{X^{*}}
\end{eqnarray}
for all $X$, $Y \in \mathcal{C}$.

If $(H, r)$ is a co-quasitriangular Hopf algebra then ribbon
structures on $\Corep (H, r)$ are in bijection with \textit{ribbon
elements} of $(H, r)$, i.e., convolution invertible central
elements $\alpha \in H^{*}$ such that $\alpha \circ S = \alpha$
and
$$
\alpha (xy) = \alpha(x_{(1)}) \alpha (y_{(1)}) (r_{21} * r)
(x_{(2)}, y_{(2)})
$$
for all $x$, $y \in H$. The ribbon structure associated to the
ribbon element $\alpha$ is
$$
\theta_{V} : V \to V, \quad v \mapsto \sum \alpha(v_{(1)}) v_{(0)}.
$$

The ribbon elements of $(H, r)$ can be determined in the following
way (see \cite[Proposition 2]{R94} where the result appears in
dual form). Let $ \eta : H \to k$, $\eta (h) = r (h_{(2)},
S(h_{(1)}))$, $h \in H$, be the Drinfeld element of $(H,
r)$. Then $\eta^{-1} (h) = r (S^{2} (h_{(2)}), h_{(1)})$, for all
$h \in H$, the element $(\eta \circ S)* \eta^{-1}$ is a group-like
element of $H^{*}$, and the map
$$
\gamma \mapsto \gamma*\eta
$$
establishes a one-to-one correspondence between the set of
group-like elements $\gamma \in H^{*}$ satisfying $\gamma^{2} =
(\eta \circ S) * \eta^{-1}$ and $S^{2}_{H^{*}} (p) = \gamma^{-1} *
p * \gamma$, for all $p \in H^{*}$, and the set of ribbon elements
of $(H, r)$.

%%%%%%%%%%%%%%%%%%%%%%%%%%%%%%%%%%%%%%%%%%%%%%%%%%%%%%%%%%%%%
\subsection{Pointed braided fusion categories}
\label{prelim: pointed fusion}

Let $\mathcal{C}$ be a pointed braided fusion category. Then the
isomorphism classes of simple objects of $\mathcal{C}$ form a
finite abelian group $\Gamma$. The braiding determines a function
$c : \Gamma \times \Gamma \to k^{\times}$ and the function $q:
\Gamma \to k^{\times}$, $q (g) = c (g, g)$, $g \in \Gamma$, is a
\textit{quadratic form} on $\Gamma$, i.e., $q (g^{-1}) = q (g)$,
for all $g \in \Gamma$. The symmetric function
$$
b (g, h) = \frac{q (gh)}{q(g) q(h)}, \qquad g, h \in \Gamma
$$
is a bicharacter on $\Gamma$. It was shown in \cite{JS93} (see
also \cite[Appendix D]{DGNO10}) that the assignment
$$
\mathcal{C} \mapsto (\Gamma,\, q)
$$
determines an equivalence between the 1-categorical truncation 
of the $2$-category of pointed braided fusion categories
and the category of pre-metric groups. The objects of
the latter category are finite abelian groups equipped with a
quadratic form, and morphisms are group homomorphisms preserving
the quadratic forms. We will denote a pointed braided fusion
category associated to $(\Gamma,\, q)$ by $\mathcal{C} (\Gamma,\, q)$.
%As a fusion category, $\C(\Gamma,\, q)$ is equivalent to the category $\Vec_\Gamma$
%of $\Gamma$-graded vector spaces.

Let $\Quad(\Gamma)$ denote the set  of quadratic forms on $\Gamma$
and let $\Quad_d(\Gamma) \subset \Quad(\Gamma)$ be the subgroup
of diagonalizable quadratic forms on $\Gamma$, i.e., such that there is a
bilinear form $r_0: \Gamma\times \Gamma \to k^\times$ with $q(g) =
r_0(g,\, g)$ for all $g\in \Gamma$ (i.e., $q$ is the restriction of
$r_0$ on the diagonal). The corepresentation categories of
co-quasitriangular pointed
semisimple Hopf algebras are precisely those equivalent to fusion categories of the
form $\C(\Gamma,\, q)$ with $q\in \Quad_d(\Gamma)$. We will use the following notation:
\[
\C(\Gamma,\, r_0) := \Corep(k[\Gamma],\, r_0).
\]

%%%%%%%%%%%%%%%%%%%%%%%%%%%%%%%%%%%%%%%%%%%%%%%%%%%%%%%%%%%%%
\subsection{The symmetric center}
\label{prelim:symcenter}

Let $\mathcal{C}$ be a braided tensor category with braiding
$\{c_{X, Y}\}_{X, Y \in \mathcal{C}}$. 
%If $\mathcal{D}$ is a
%tensor subcategory of $\mathcal{C}$, then its \textit{centralizer}
%$\mathcal{D}'$ is the full subcategory of $\mathcal{C}$ consisting
%of those objects $Y$ such that $c_{Y,X} \circ c_{X, Y} = \id_{X
%\ot Y}$, for all $X \in \mathcal{C}$. 
The {\em symmetric center} $\Z_{sym}(\C)$ of $\C$ is the full tensor subcategory of $\C$ 
consisting of all objects $Y$ such that $c_{Y,X} \circ c_{X, Y} = \id_{X\ot Y}$ for all $X \in \C$. 

A braided tensor category $\C$ is called {\em symmetric} if $\Z_{sym}(\C) = \C$. Any symmetric
fusion category $\C$ has a canonical (possibly trivial) $\mathbb{Z}/2\mathbb{Z}$-grading
\begin{equation}
\label{ Z2 grading on symmetric}
\C =\C_+ \oplus \C_-,
\end{equation}
where $\C_+$ is the maximal Tannakian subcategory of $\C$ \cite{De02}. 
In terms of the canonical ribbon element $\theta$ of $\C$, one has $\theta_X =\pm  \id_X$
when $X\in \C_\pm$. 

A braided tensor category $\C$  is called {\em factorizable} if $\Z_{sym}(\C)$ is trivial. 
%In this case there is a canonical braided tensor equivalence $\Z(\C) \cong \C \bt \C^\rev$,
%where $\C^\rev$ stands for $\C$ with the reverse braiding. 

%There is a pair of canonical braided tensor embeddings
%$\C \to \Z(\C)$ given by
%\begin{equation}
%\label{bremb}
%X \mapsto (X,\, c_{X, -}) \quad \text{ and } \quad X \mapsto (X,\, c_{-, X}^{-1}).
%\end{equation}
%Let $\C_\pm$ denote the images of these embeddings. The {\em symmetric center} of $\C$
%is the $\Z_{sym}(\C) = \C_+\cap \C_-$.   Clearly, $\Z_{sym}(\C)$ is identified with a symmetric tensor
%subcategory of $\C$. In fact, $\C$ is symmetric if and only if $\C =\Z_{sym}(\C)$.
%A braided tensor category $\C$ is {\em factorizable} if $\Z_{sym}(\C)$ is trivial.

Let  $\C =\Corep (H)$, where $H$ is a co-quasitriangular  Hopf
algebra with an $r$-form $r$ (so that the braiding of $\C$ is
given by \eqref{braiding from r-form}). Then
$\Z_{sym}(\C)=\Corep(H_{sym})$  for a Hopf subalgebra $H_{sym}
\subset H$. The category $\C$ is symmetric if  and only if
$H_{sym}=H$ and it is factorizable if  and only if $H_{sym} =k1$.
This Hopf subalgebra $H_{sym} $ was described by Natale in
\cite{Na06} (note that in \cite{Na06} quasitriangular Hopf algebras
were considered while we deal with the dual situation).  Below we
reproduce this description using our terminology.

Consider the linear map $\Phi_r: H \to H^*$ given by
\[
\Phi_r(x)(y)= r(y\1,\, x\1)\, r
(x\2,\,y\2),\qquad x,y\in H.
\]
Its image $\Phi_r(H)$ is a normal coideal subalgebra of $H^*$.
Hence, $H^*\Phi_r(H)^+$ is a Hopf ideal of $H^*$. We have
$H_{sym}= (H^*\Phi_r(H)^+)^\perp$. Here for $I\subset H^*$ we
denote $I^\perp\subset H$ its annihilator, i.e, $I^\perp= \{ x\in
H \mid f (x)=0 \textnormal{  for all } f \in I\}$. Explicitly,
\begin{equation} \label{Hsym formula}
H_{sym} = \{x \in H \mid x_{(1)} r (x_{(2)}, y_{(1)}) r (y_{(2)},
x_{(3)}) = \varepsilon(y)x, \textnormal{ for all } y \in H \},
\end{equation}
Equivalently, $H_{sym}$ consists of all $x\in H$ such that the
squared braiding $c_{H,H}^2:  H\ot H \to H\ot H$ fixes $x \ot y$
for all $y\in H$.

%%%%%%%%%%%%%%%%%%%%%%%%%%%%%%%%%%%%%%%%%%%%%%%%%%%%%%%%%%%%%
\subsection{The Drinfeld center of a tensor category and Yetter-Drinfeld modules}

An important example of a braided tensor category is the {\em Drinfeld center}
$\mathcal{Z} (\mathcal{C})$ of a finite tensor category
$\mathcal{C}$. The objects of $\mathcal{Z} (\mathcal{C})$ are
pairs $(Z, \gamma)$ consisting of an  object $Z$ of $\mathcal{C}$
and a natural isomorphism $\gamma : - \ot Z \to Z \ot -$  satisfying 
a hexagon axiom.
%such
%that
%$$
%(\gamma_{X} \ot \id_{Y}) \circ a_{X, Z, Y}^{-1} \circ (\id_{X} \ot
%\gamma_{Y}) = a_{Z, X, Y}^{-1} \circ (\gamma_{X \ot Y}) \circ
%a_{X, Y, Z}^{-1}
%$$
%for all $X$, $Y \in \mathcal{C}$. 

The braiding of $\mathcal{Z}(\mathcal{C})$ is
$$
c_{(Z, \gamma), (Z', \gamma')} = \gamma'_{Z} : (Z, \gamma) \ot
(Z', \gamma') \to (Z', \gamma') \ot (Z, \gamma).
$$

\begin{example}
If $H$ is a finite dimensional Hopf algebra then the Drinfeld center of
$\Rep (H)$ is braided equivalent to $\Rep \big( D(H) \big)$, where
$D (H)$ is the \textit{Drinfeld double} of $H$. As a coalgebra, $D
(H) = H^{*\text{cop}} \ot H$, where $H^{*\text{cop}}$ is the
co-opposite dual of $H$, while the algebra structure is given by
$$
(p \ot h) (p' \ot h') = p \left( h_{(1)} \rightharpoonup p'
\leftharpoonup S^{-1} \big( h_{(3)} \big) \right) \ot h_{(2)} h',
\quad p, p' \in H^{*}, \,\, h, h' \in H,
$$
where $(h \rightharpoonup p \leftharpoonup g) (x) = p (g x h)$,
for all $h$, $g$, $x \in H$, $p \in H^{*}$. 
%The braiding of $\Rep
%\big( D(H) \big)$ corresponds to the canonical $R$-matrix on $D(H)$:
%$$
%\mathcal{R} =\sum_i\, (\varepsilon \ot e_i) \ot \big( e_{i}^{*}
%\ot 1 \big),
%$$
%where $\{e_{i}\}$ is a basis of $H$ and $\{e_{i}^{*}\}$ is the
%corresponding dual basis of $H^{*}$.
\end{example}

It is well-known that $\Rep \big( D(H) \big)$ is braided
equivalent to the category $^H _{H} \mathcal{YD}$ of
\textit{(left) Yetter-Drinfeld modules} over $H$. An object $V$ in
this category has simultaneously a structure of a left $H$-module,
$h \ot v \mapsto h \cdot v$, and a structure of a left
$H$-comodule $\delta : V \to H \ot V$, $\delta (v) = v_{(-1)} \ot
v_{(0)}$, such that the following condition is satisfied:
$$
\delta (h \cdot v) = h_{(1)} v_{(-1)} S (h_{(3)}) \ot h_{(2)}
\cdot v_{(0)}, \qquad h \in H, \, v \in V.
$$
Morphisms between such objects are linear maps preserving both the
action and the co-action of $H$. The braiding of $^H _{H}
\mathcal{YD}$ is given by:
$$
c_{U, V} : U \ot V \to V \ot U, \quad u \ot v \mapsto u_{(-1)}
\cdot v \ot u_{(0)}
$$
for all $u \in U$, $v \in V$ and $U$, $V \in \, ^H _{H}
\mathcal{YD}$.

For a finite group $\Gamma$ we denote $^{k[\Gamma]} _{k[\Gamma]}
\mathcal{YD}$ by $^{\Gamma} _{\Gamma} \mathcal{YD}$. If $\Gamma$
is abelian then
$$
^{\Gamma} _{\Gamma} \mathcal{YD} \simeq \mathcal{Z}
(\Vec_{\Gamma}) \simeq \C(\Gamma\times \widehat{\Gamma},\, h),
$$
where $h : \Gamma \times \widehat{\Gamma} \to k^{\times}$ is the
canonical quadratic form, $h (g, \chi) = \chi (g),\,g \in
\Gamma,\, \chi \in \widehat{\Gamma}$.

%If $\alpha : \Gamma \to \Gamma'$ is a group homomorphism and $V
%\in \, ^{\Gamma} _{\Gamma} \mathcal{YD}$ then $V \in \, ^{\Gamma'}
%_{\Gamma'} \mathcal{YD}$, via the action $\cdot'$ and coaction
%$\delta'$ given by
%$$
%h' \cdot' v = \alpha^{-1} (h') \cdot v \qquad \textnormal{ and }
%\qquad \delta' (v) = (\alpha \ot \id) \delta (v)
%$$
%for all $v \in V$ and $g' \in \Gamma'$. We denote $V$, endowed
%with this structure of Yetter-Drinfeld module over $\Gamma'$, by
%$\ind_{\alpha} (V)$.

\begin{remark}
\label{braided TC in center}
For a braided tensor category $\C$ there exist canonical braided tensor embeddings
\begin{equation}
\label{hook}
\C \hookrightarrow \Z(\C) : X \mapsto (X,\, c_{-,X})\quad \mbox{and} \quad \C^\rev \hookrightarrow \Z(\C) : X \mapsto (X,\, c_{X,-}^{-1}) .
\end{equation}
The intersection of images of $\C$ and $\C^\rev$ in $\Z(\C)$ is equivalent to $\Z_{sym}(\C)$, the symmetric center of $\C$. 
\end{remark}

%%%%%%%%%%%%%%%%%%%%%%%%%%%%%%%%%%%%%%%%%%%%%%%%%%%%%%%%%%%%%%%%%%
\subsection{2-cocycles and deformations}
%%%%%%%%%%%%%%%%%%%%%%%%%%%%%%%%%%%%%%%%%%%%%%%%%%%%%%%%%%%%%%%%%%
\label{prelim 2-cocycles}

A \textit{ $2$-cocycle} on $H$ is a convolution invertible
linear map $\sigma: H \ot H \to k$ such that $\sigma (x, 1) =
\varepsilon (x) = \sigma(1, x)$ and
\begin{equation}
\sigma(x\1, y\1) \sigma(x\2 y\2, z) = \sigma(y\1, z\1) \sigma(x, y\2 z\2)
\end{equation}
for all $x$, $y$, $z \in H$. Two $2$-cocycles $\sigma$ and
$\sigma'$ are \textit{gauge equivalent} if there exists a
convolution invertible map $u: H \to k$ such that
$$
\sigma' (x, y) = u^{-1} (x\1) u^{-1} (y\1) \sigma(x\2, y\2) u (x\3
y\3),\qquad x,y\in H.
$$
for all $x$, $y \in H$.

The isomorphism classes of fiber functors on $\Corep (H)$ are in bijection with the set of
gauge equivalence classes of 2-cocycles on $H$.

Twisting the multiplication of $H$ on both sides by a $2$-cocycle
$\sigma$, we obtain a new Hopf algebra, denoted $H^{\sigma}$ and
called a \textit{cocycle deformation} of $H$. We have $H^{\sigma}
= H$ as a coalgebra and the multiplication of $H^{\sigma}$ is
given by
\begin{equation}
\label{sigma on both sides} x \cdot_\sigma y = \sigma (x\1, y\1)
x\2 y\2 \sigma^{-1} (x\3, y\3), \qquad x, y \in H.
\end{equation}
If $H$ is a co-quasitriangular Hopf algebra with an $r$-form $r$
then $H^{\sigma}$ is also co-quasitriangular  with $r$-form
$r^{\sigma}$, given by
\begin{equation}
\label{rsigma}
r^{\sigma} (x, y) = \sigma (y\1, x\1) r (x\2, y\2)
\sigma^{-1} (x\3, y\3)  \qquad x, y \in H.
\end{equation}
For gauge equivalent $2$-cocycles $\sigma$ and
$\sigma'$  co-quasitriangular Hopf algebras
 $H^\sigma$ and  $H^{\sigma'}$ are isomorphic.

A $2$-cocycle $\sigma$ on $H$ is called \textit{invariant} if
\begin{equation} \label{invariant 2cocycle}
\sigma(x\1, y\1) x\2 y\2 = x\1 y\1 \sigma(x\2, y\2)
\end{equation}
for all $x,\,y\in H$. Note that  $\sigma$ is invariant if and only if $H^\sigma =
H$ as Hopf algebras. The set of invariant $2$-cocycles is a group
under convolution product denoted by $\Zinv (H)$.

For a convolution invertible linear map $u : H \to k$ such that $u
(1) = 1$ and $u (x_{(1)}) x_{(2)} = x_{(1)} u (x_{(2)})$, for all $x \in
H$, the map
$$
\sigma_{u}: H \ot H \to k, \quad  \sigma_{u} (x, y) = u (x_{(1)}) u
(y_{(1)}) u^{-1} (x_{(2)}y_{(2)}), \quad x, y \in H
$$
is an invariant $2$-cocycle. The set of all such $2$-cocycles is a
subgroup of $\Zinv (H)$ denoted by $\Binv(H)$.

The \textit{second invariant cohomology group} of $H$ \cite{BC06}
is the quotient group
$$
\Hinv (H) = \Zinv (H) / \Binv (H).
$$
For example, if $\Gamma$ is a group, then $\Hinv (k[\Gamma]) =
\textnormal{H}^{2} (\Gamma,\, k^{\times})$, the second cohomology group
of $\Gamma$ with coefficients in $k^\times$.

%%%%%%%%%%%%%%%%%%%%%%%%%%%%%%%%%%%%%%%%%%%%%%%%%%%%%%%%%%%%%%%%%%
\subsection{Galois objects and 2-cocycles}
%%%%%%%%%%%%%%%%%%%%%%%%%%%%%%%%%%%%%%%%%%%%%%%%%%%%%%%%%%%%%%%%%%
\label{Galois objects}

We recall here the connection between Galois objects and
$2$-cocycles. Let $H$ be a Hopf algebra.

A \textit{left} $H$-\textit{Galois object} is a non-zero left
$H$-comodule algebra $A$ such that $A^{\textnormal{co} H} = k$ and
the linear map $A \ot A \to H \ot A$, $a \ot b \mapsto a_{(-1)}
\ot a_{(0)}b$, for all $a$, $b \in A$, is bijective.

If $\sigma$ is a $2$-cocycle on $H$ then $H$, with the comodule
structure given by $\Delta$ and multiplication
\begin{equation} \label{sigma on left} x\cdot y = x_{(1)}y_{(1)}
 \sigma^{-1} (x_{(2)},\, y_{(2)}), \qquad x, y\in H,
\end{equation}
is a left $H$-Galois object, denoted by $H_{\sigma^{-1}}$.

Conversely, if $A$ is a left $H$-Galois object, then there exists
a left $H$-colinear isomorphism $\psi: H \to A$ such that $\psi
(1) = 1$. The map $\kappa : H \ot H \to k$, defined by
\begin{equation} \label{cocycle formula}
\kappa (x, y) = \varepsilon \Big( \psi^{-1} \big( \psi(x) \psi(y)
\big) \Big), \qquad x, y \in H
\end{equation}
is convolution invertible, $\sigma := \kappa^{-1}$ is a $2$-cocycle
and $\psi: \, H_{\sigma^{-1}} \to A$ is a left $H$-comodule
algebra isomorphism.

%A \textit{right} $H$-\textit{Galois object} is a non-zero right
%$H$-comodule algebra $A$ such that $A^{\textnormal{co} H} = k$ and
%the linear map $A \ot A \to A \ot H$, $a \ot b \mapsto ab_{(0)}
%\ot b_{(1)}$, for all $a$, $b \in A$, is bijective.

%If $\sigma$ is a $2$-cocycle on $H$ then $H$, with the comodule
%structure given by $\Delta$, and multiplication
%\begin{equation} \label{sigma on left} x\cdot y =
%\sigma(x\1,\, y\1) x\2 y\2,\qquad x,y\in H,
%\end{equation}
%is a right $H$-Galois object, denoted by $_{\sigma} H$.

%Conversely, if $A$ is a right $H$-Galois object, then there exists
%a right $H$-colinear isomorphism $\psi: H \to A$ such that $\psi
%(1) = 1$. The map $\sigma : H \ot H \to k$, defined by
%\begin{equation}
%\sigma (x, y) = \varepsilon \Big( \psi^{-1} \big( \psi(x) \psi(y)
%\big) \Big), \qquad x, y \in H
%\end{equation}
%is a $2$-cocyle and $\psi: \, _{\sigma} H \to A$ is a right
%$H$-comodule algebra isomorphism.

%%%%%%%%%%%%%%%%%%%%%%%%%%%%%%%%%%%%%%%%%%%%%%%%%%%%%%%%%%%%%%%%%%%%%%%%%%
%%%%%%%%%%%%%%%%%%%%%%%%%%%%%   2-COCYCLES     %%%%%%%%%%%%%%%%%%%%%%%%%%%%%%%%%%%
%%%%%%%%%%%%%%%%%%%%%%%%%%%%%%%%%%%%%%%%%%%%%%%%%%%%%%%%%%%%%%%%%%%%%%%%%%%%%%%%%%%

\section{Quantum linear spaces of symmetric type}
\label{Sect:QLS}

%%%%%%%%%%%%%%%%%%%%%%%%%%%%%%%%%%%%%%%%%%%%%%%%%%%%%%%%%%%%%%%%%%%%%%
\subsection{Quantum linear spaces}
%%%%%%%%%%%%%%%%%%%%%%%%%%%%%%%%%%%%%%%%%%%%%%%%%%%%%%%%%%%%%%%%%%%%%%

An important class of pointed Hopf algebras with a given abelian
group $\Gamma$ of group-like elements can be constructed as
follows \cite{AS98}.

Let $g_{1}, \dots, g_{n}$ be elements of $\Gamma$ and let
$\chi_{1}, \dots, \chi_{n}$ be elements of the dual group
$\widehat{\Gamma}$ such that
$$
\chi_{i} (g_{i}) \neq 1 \quad \textnormal{ and } \quad \chi_{j}
(g_{i}) \chi_{i} (g_{j}) = 1
$$
for all $i,j=1,\dots,n$, $i\neq j$. 

\begin{definition}
\label{QLS def}
A \textit{quantum linear
space} associated to the above datum $(g_{1}, \dots, g_{n}, \chi_{1},
\dots, \chi_{n})$ is the Yetter-Drinfeld module
\begin{equation}
\label{QLS}
V = \bigoplus_{i = 1}^{n} k x_{i} \in \, \YD,
\end{equation}
with $h \cdot x_{i} = \chi_{i} (h) x_{i}$, for all $h \in \Gamma$,
and $\rho (x_{i}) = g_{i} \ot x_{i}$, for all $i$.
\end{definition}

Let $V_g^\chi$ denote the simple object in $\YD$ corresponding to $g\in G$ and  $\chi\in \widehat\Gamma$.
Then $x_i \in V_{g_i}^{\chi_i},\, i=1,\dots,n$. 

The braiding on $V\ot V$ on the basic elements $x_i\ot x_j,\, i,j=1\dots, n, $ is given by
\begin{equation}
\label{cVV}
c_{V\ot V} (x_i \ot x_j) = \chi_j(g_i)\, x_j \ot x_i.
\end{equation}

\begin{definition}
We will say that a quantum linear space $V$ is of {\em symmetric
type} if $c_{V,V}^2 =\id_{V\ot V}$ (i.e., $\chi_i(g_i)=-1$ for all $i=1,\dots,n$).
\end{definition}

\begin{remark}
Equivalently, a quantum linear space of symmetric type is an object $V\in \YD$
such that $c_{V,V}^2 =\id_{V\ot V}$ and  $\theta_V=-\id_V$, where $\theta$ is the canonical
ribbon element of $\YD$. Note that this definition does not depend on the choice
of ``basis" $g_i,\chi_i,\, i=1,\dots, n$.
\end{remark}

Note that the transposition map
\begin{equation}
\label{transposition}
\tau_{V,V} : V\ot V \to V\ot V : v_1\ot v_2 \mapsto v_2\ot v_1,\qquad v_1,\, v_2\in V,
\end{equation}
is a morphism in $\YD$. 

\begin{lemma}
\label{symmetry}
Let $V\in \YD$ be a quantum linear space of symmetric type and $r:V\ot V \to k$
be a morphism in $\YD$.  Then $r\circ c_{V,V} =  - r \circ \tau_{V,V}$.
\end{lemma}
\begin{proof}
It suffices to check that $ c_{U,U^*} =  -  \tau_{U,U^*}$ for every simple object  $U=V_{g_i}^{\chi_i} \subset V$
such that $U^*$ is also a subobject of $V$.   In this case  $(g_i,\, \chi_i) = (g_j^{-1},\,x_j^{-1})$
for some $j$.  Therefore,
\[
c_{U,\, U^*} (y \ot y') = \chi_j(g_i) y'\ot y = \chi_i^{-1}(g_i)  y'\ot y = - y'\ot y
\]
for all $y \in U,\, y' \in U^*$, as required .
%Hence, $ r|_{U_i^*\ot U_i} \circ c_{U_i,\, U_i^*}  = - r|_{U_i^*\ot U_i} \circ \tau_{U_i,\, U_i^*}$
%which implies the result. 
\end{proof}

Given a quantum linear space $V$ we associate to it the
bosonization  $\mathfrak{B} (V) \# k [\Gamma]$ of the Nichols
algebra $\mathfrak{B} (V)$ by $k[\Gamma]$\footnote{We will abuse the terminology and will
also refer to  $\mathfrak{B} (V) \# k [\Gamma]$ as a quantum linear space.}. 
This Hopf algebra is generated by the group-like elements $h \in \Gamma$ and the
$(g_{i}, 1)$-skew primitive elements $x_{i}$ (i.e., such that $\Delta (x_{i})
= g_{i} \ot x_{i} + x_{i} \ot 1$), $i = 1, \dots, n$, satisfying
the following relations:
$$
hx_{i} = \chi_{i} (h) x_{i}h, \qquad  \quad x_{i}^{h_{i}} = 0,
\qquad h \in \Gamma, \,\, i = 1, \dots, n,
$$
$$
x_{i}x_{j} = \chi_{j} (g_{i}) x_{j} x_{i}, \qquad i,\,j=1,\dots, n, 
$$
where $h_{i}$ is the order of the root of unity $\chi_{i}
(g_{i})$. The set 
\[
\{ g x_{1}^{i_{1}} \cdots x_{n}^{i_{n}} \mid g \in \Gamma, \, 0 \leq
i_{j} < h_{j}, \,\, j=1, \dots, n \}
\]
is a basis of $\mathfrak{B} (V) \# k [\Gamma]$.  

If $V$ is a quantum linear space then the liftings of
$\mathfrak{B} (V) \# k [\Gamma]$, i.e., the pointed Hopf algebras
$H$ for which there exists a Hopf algebra isomorphism
\[
\gr \, H \simeq \mathfrak{B} (V) \# k [\Gamma],
\]
where $\gr \, H$ is the graded Hopf algebra associated to the
coradical filtration of $H$, were classified  in \cite[Theorem
5.5]{AS98}.  Namely, for any such a lifting $H$, there exist
scalars $\mu_i \in \{0,\,1\}$ and $\lambda_{ij} \in k\, (1\leq i <
j \leq n)$, such that
\begin{enumerate}
\item[] $\mu_i$ is arbitrary  if $g_i^{h_i} \neq 1$ or $\chi^{r_i}
=1$, and $\mu_i=0$ otherwise, \item[] $\lambda_{ij}$ is arbitrary
if $g_ig_j\neq 1$ and $\chi_i\chi_j =1$, and $\lambda_{ij}=0$
otherwise.
\end{enumerate}
The Hopf algebra $H$ is generated by the group-like elements $g
\in \Gamma$ and the $(g_{i}, 1)$-skew-primitive elements $a_i,\,
i=1,\dots,n$, such that
\begin{eqnarray*}
ga_i &=& \chi_i(g) a_i g,\qquad a_i^{h_i} =\mu_i(1-g_i^{h_i}),\qquad i=1,\dots,n,\\
a_i a_j &=& \chi_j(g_i) a_j a_i +\lambda_{ij}(1-g_ig_j),\qquad
1\leq i < j \leq n.
\end{eqnarray*}

It was shown in \cite{Ma01} that these liftings are
cocycle deformations of $\mathfrak{B} (V) \# k [\Gamma]$ (see Section~\ref{prelim 2-cocycles}).

%\begin{remark}
%The above lifting Hopf algebra $H$ is isomorphic to $\mathfrak{B} (V) \# k [\Gamma]$ if and only if
%$\mu_i=\lambda_{ij}=0$ for all $i,j=1,\dots,n$. Indeed,  if some of the parameters are different from $0$
%then $H$ is not coradically graded.
%\end{remark}

\begin{remark} \label{xP remark}
Suppose that $V$ is a quantum linear space of symmetric type. Then
$x_{i}^{2} = 0$ for all $i = 1, \dots, n$. For a subset $P =
\{i_{1}, i_{2}, \dots, i_{s}\} \subseteq \{1, \dots, n\}$ such
that $i_{1} < i_{2} < \cdots <i_{s}$ we denote the element
$x_{i_{1}} \cdots x_{i_{s}}$ by $x_{P}$ and use the convention
that $x_{\emptyset} = 1$. Clearly, the set  $\{ gx_{P} \mid g \in \Gamma, \,\, P \subseteq
\{1, \dots, n\} \}$ is a basis of $\mathfrak{B} (V) \# k[\Gamma]$.

Let $F \subseteq P$ be subsets of $\{1, \dots, n\}$ and let $\psi
(P, F)$ be the element of $k$ such that $x_{P} = \psi (P, F) x_{F}
x_{P \setminus F}$. Thus,
\begin{equation}
\psi (P, F) = \prod_{\substack{j \in F, \, i \in P \setminus F \\
i < j}} \chi_{j} (g_{i}) \label{psi}.
\end{equation}
It is easy to check that the comultiplication formula for $x_{P}$
is given by
\begin{equation}
\label{delta xP} \Delta (x_{P}) = \sum_{F \subseteq P} \psi (P, F)
g_{F} x_{P \setminus F} \ot x_F,
\end{equation}
where $ g_F = \Pi_{i\in F}\, g_i$ and $g_{\emptyset} = 1$.
\end{remark}

%\begin{remark}
%Suppose that $\chi_{j} (g_{i}) = -1$ for all $i$, $j \in \{1,
%\dots, n\}$. If $P = \{i_{1}, i_{2}, \dots, i_{s}\}$ with $i_{1} <
%i_{2} < \cdots < i_{s}$, and $F = \{i_{u_{1}}, \dots, i_{u_{r}}\}$
%is a subset of $P$, then
%$$
%\psi (P, F) = (-1)^{u_{1} + u_{2} + \cdots + u_{r} - \frac{r (r +
%1)}{2}}
%$$
%and
%$$
%\Delta (x_{P}) = \sum_{F = \{ i_{u_{1}}, \dots, i_{u_{r}} \}
%\subseteq P} (-1)^{u_{1} + u_{2} + \cdots + u_{r} - \frac{r (r +
%1)}{2}} g_{F} x_{P \setminus F} \ot x_F
%$$
%The above formula for co-multiplication appears in \cite{CD99}
%where the authors consider a special class of pointed Hopf
%algebras, namely that of Nichols Hopf algebras.
%\end{remark}

%%%%%%%%%%%%%%%%%%%%%%%%%%%%%%%%%%%%%%%%%%%%%%%%%%%%%%%%%%%%%%%%%%%%%%

\subsection{The double of a quantum linear space} 
\label{double of quantum linear space}

We introduce here a construction which will appear in Section~\ref{Sect: Drinfeld center}
when we discuss the adjoint subcategory of the center of a pointed
braided finite tensor category.

Let $V \in \YD$ be the  quantum linear space of symmetric type
associated to a datum $(g_{1}, \dots, g_{n}, \chi_{1}, \dots,
\chi_{n})$. Let $\Sigma$ be the subgroup of $\Gamma \times
\widehat{\Gamma}$ generated by $(g_{i},\, \chi_{i}^{-1})$, $i = 1,
\dots, n$, and define characters $\varphi_{i} : \Sigma \to k^{\times}$ by
\[
\varphi_{i} (g,\, \chi) = \chi_{i} (g),\quad \mbox{for all } (g,\, \chi) \in \Sigma, \quad i = 1, \dots, n.
\]
We have
$$
\varphi_{i} (g_{i}, \chi_{i}^{-1}) = -1 \quad \textnormal{ and }
\quad \varphi_{j} (g_{i}, \chi_{i}^{-1}) \varphi_{i} (g_{j},
\chi_{j}^{-1}) = 1
$$
for all $i, j = 1, \dots, n$. Thus we can consider  the quantum
linear space of symmetric type $W \in \, ^{\Sigma}_{\Sigma} \mathcal{YD}$ associated
to the datum $\big( (g_{1}, \chi_{1}^{-1}), \dots, (g_{n},
\chi_{n}^{-1}), \varphi_{1}, \dots, \varphi_{n} \big)$.

\begin{definition}
\label{double of QLS}
We call the quantum linear space $D(V):= W \oplus W^{*} \in \,
^{\Sigma}_{\Sigma} \mathcal{YD}$ the \textit{Drinfeld double} of $V$.
\end{definition}

Note that the  quantum linear space $D(V)$ is of symmetric type.

There is a canonical bilinear form  $r_{D(V)} : D(V) \ot D(V) \to k$
given by
\begin{equation}
\label{rDV}
r_{D(V)}((w,\, f) \ot (w',\, f')) :=   \ev_{W} (f \ot w') +
\ev_{W} c_{W, W^{*}} (w, f')
\end{equation}
for all $w, w' \in W, \, f, f' \in W^{*}$,
where  $\ev_W: W^*\ot W\to k$ is the evaluation morphism
and $c_{W,W^*}: W \ot W^*\to W^*\ot W$ is braiding
in $^{\Sigma}_{\Sigma} \mathcal{YD}$.

Note that $r_{D(V)}$ is  a symplectic bilinear form on $W\oplus W^*$.
Indeed, if $\{x_{i}\}$ is a basis of $W$ such that $x_{i} \in
W_{(g_{i}, \chi_{i}^{-1})}^{\varphi_{i}}$, $i = 1, \dots, n$, and
$\{x_{i}^{*}\}$ is the dual basis of $W^*$, then the matrix of
$r_{D(V)}$ with respect to the basis $(x_{1}, \dots, x_{n},
x_{1}^{*}, \dots, x_{n}^{*})$ is
$$
\left(\begin{array}{cc}
0 & -I_{n}\\
I_{n} & 0
\end{array} \right),
$$
where $I_{n}$ denotes the $n$-by-$n$ identity matrix.

%%%%%%%%%%%%%%%%%%%%%%%%%%%%%%%%%%%%%%%%%%%%%%%%%%%%%%%%%%%%%%%%%%%%%%
\subsection{Galois objects for quantum linear spaces}
%%%%%%%%%%%%%%%%%%%%%%%%%%%%%%%%%%%%%%%%%%%%%%%%%%%%%%%%%%%%%%%%%%%%%%
\label{Mombelli's classification}

Let $V\in \YD$ be a quantum linear space of symmetric type and let
$H =\mathfrak{B} (V) \# k [\Gamma]$.  In \cite{Mo11} Mombelli
classified equivalence classes of  exact indecomposable
$\Rep(H)$-module categories.  In particular, he classified
$H$-Galois objects and, hence, 
$2$-cocycles on $H$. Here we recall this classification, see
\cite[Section 4]{Mo11}  for details.

A typical $H$-Galois object is determined by a
$2$-cocycle $\psi \in Z^{2} (\Gamma,\, k^{\times})$ and two
families of scalars $\xi =(\xi_{i})_{i = 1, \dots, n}$ and $\alpha
= (\alpha_{ij})_{1 \leq i <j \leq n}$, satisfying:
\begin{align}
\label{Mombelli's xi}
\xi_{i} & = 0  \quad  \textnormal{ if }  \chi_{i}^{2} (g) \neq \frac{ \psi (g ,\, g_{i}^{2}) }{ \psi (g_{i}^{2}, g) },\\
\label{Mombelli's alpha}
\alpha_{ij} & = 0 \quad  \textnormal{ if }  \chi_{i} \chi_{j} (g) \neq \frac{ \psi (g,\, g_{i}g_{j}) }{\psi (g_{i}g_{j},\, g)},
\end{align}
for all $g \in \Gamma$. To this datum one assigns a left
$H$-comodule algebra $\mathcal{A} (\psi, \xi, \alpha)$ generated
as an algebra by $\{e_{g}\}_{g \in \Gamma}$ and $v_{1}, \dots,
v_{n}$ subject to the relations:
\begin{align*}
e_{f} e_{g} & = \psi (f,\, g) e_{fg}, \qquad f, g \in \Gamma\\
e_{f} v_{i} & = \chi_{i} (f) v_{i} e_{f}, \qquad f \in \Gamma, \,\, i = 1, \dots, n \\
v_{i}v_{j} - \chi_{j} (g_{i}) v_{j} v_{i} & =   \alpha_{ij} e_{g_{i}g_{j}}, \qquad 1 \leq i < j \leq n  \\
v_{i}^{2} & =   \xi_{i} e_{g_{i}^{2}}, \qquad i = 1, \dots, n
\end{align*}
The left $H$-comodule structure of $\mathcal{A} (\psi, \xi,
\alpha)$ is $\lambda : \mathcal{A} (\psi, \xi, \alpha) \to H \ot
\mathcal{A} (\psi, \xi, \alpha)$,
\[
\lambda (v_{i}) = g_{i} \ot v_{i} + x_{i} \ot 1 \qquad
\textnormal{and} \qquad \lambda (e_{f}) = f \ot e_{f}
\]
for all $i = 1, \dots, n$ and $f\in \Gamma$. Two $H$-Galois
objects $\mathcal{A} (\psi, \xi, \alpha)$ and $\mathcal{A}
(\psi', \xi', \alpha')$ are isomorphic if and only $\psi,\, \psi'$
are cohomologous, $\xi=\xi'$, and $\alpha=\alpha'$.

\begin{remark}
\label{sigma from A} Suppose that $\psi = 1$. Then the $2$-cocycle
$\sigma$ corresponding to the $H$-Galois object $\mathcal{A} (1,
\xi, \alpha)$ above satisfies
\[
\sigma(x_i,\, x_j)-\chi_j(g_i) \sigma(x_j,\, x_i) =
\alpha_{ij},\quad\text{ and }\quad  \sigma(x_i,\, x_i)
=\xi_i,\qquad 1\leq i < j \leq n.
\]
In this case  the conditions \eqref{Mombelli's xi} and \eqref{Mombelli's alpha} are equivalent to
$\sigma -\sigma\circ c_{V,V} :V \ot V \to k$ being a $\Gamma$-module map.
\end{remark}

\subsection{$2$-cocycles on quantum linear spaces of symmetric type}
\label{2cocycles in symmetric case}

Let $V\in \YD$ be  a quantum linear  space  of symmetric type with  braiding $c_{V,V}:V\ot V \to V \ot V$
and  let  $b: V \ot V \to k$ be a bilinear form. 

Let $\tau_{V, V} : V\ot V \to V\ot V$ denote the transposition map.

\begin{definition}
\label{braiding-}
We will say that $b$ is {\em symmetric} (respectively,
{\em alternating}) if $b\circ \tau_{V,V} = b$  (respectively, $b\circ \tau_{V,V} = -b$).
\end{definition}

There is a canonical  decomposition  of $b$ into the sum of symmetric and alternating parts:
\begin{equation}
\label{sym and alt}
b = b_{sym} + b_{alt},
\end{equation}
where $b_{sym} = \frac{1}{2}  ( b +b\circ \tau_{V,V})$ and $b_{alt} = \frac{1}{2} ( b - b\circ \tau_{V,V})$.

%In this paper we  consider braided vector spaces that are objects in the braided fusion category
%$\YD$, where $\Gamma$ is a finite group.  
We will denote $\Sym^2_{\YD}(V^*)$ (respectively,
$\Alt^2_{\YD}(V^*)$) the spaces of  symmetric (respectively,  alternating) morphisms $V\ot V \to k$ in $\YD$.
We have
\[
\Sym^2_{\YD}(V^*) \subset S^2(V^*),\quad \Alt^2_{\YD}(V^*) \subset \wedge^2(V^*),
\]
where $S^2(V^*)$ and $\wedge^2(V^*)$ are spaces of usual symmetric and alternating bilinear forms
on the vector space $V$.

Let $H= \mathfrak{B} (V) \# k [\Gamma]$.

Let $\sigma: H \ot H \to k$ be a $2$-cocycle on $H$
such that $\sigma|_{\Gamma\times \Gamma} =1$ and set
\begin{equation}
\label{alpha}
\alpha =  \frac{1}{2}\sigma|_{V\ot V}  \circ (\id_{V\ot V} -c_{V\ot V}).
\end{equation}
Denote
\[
\alpha_{ij} :=  \alpha(x_i,\,x_j) = \frac{1}{2}\,\left( \sigma(x_i,\, x_j) - \chi_j(g_i) \sigma(x_j,\, x_i) \right),\qquad i,j=1,\dots,n.
\]
By Remark~\ref{sigma from A},  $\alpha: V \ot V \to k$ is a  $\Gamma$-module map and $\alpha \circ c_{V,V} =-\alpha$.

We have
\begin{equation}
\label{alpha ij and alpha ji}
\alpha_{ij} = -\chi_j(g_i) \alpha_{ji}
\end{equation}
and
\begin{equation}
\label{alpha ij=0}
\alpha_{ij} = 0  \text{ if } \chi_i\chi_j \neq \eps,\qquad i,j=1,\dots,n.
\end{equation}
%i.e., $\alpha: V \ot V \to k$ is a braiding-alternating $\Gamma$-module map.

Conversely, given $\alpha: V \ot V \to k$ such that
 scalars $\{\alpha_{ij} =\alpha(x_i,\, x_j)\}_{i,j=1}^{n}$ satisfy conditions
\eqref{alpha ij and alpha ji} and \eqref{alpha ij=0}
it follows from Remark~\ref{sigma from A}
that there is a $2$-cocycle $\sigma$ on $H$ such that
\[
\sigma|_{\Gamma \times \Gamma} =1\quad \text{and} \quad  
\sigma|_{V\ot V}  \circ (\id_{V\ot V} -c_{V\ot V}) =\alpha.
\]
Such a $2$-cocycle $\sigma$ is unique up to a gauge equivalence.

We summarize these results in the following statement.

\begin{proposition}
\label{all twists on H}
The map $\sigma \mapsto \sigma|_{V\ot V}  \circ (\id_{V\ot V} -c_{V\ot V})$ is a bijection between the set of
gauge equivalence classes of $2$-cocycles on $H$ whose restriction
on $\Gamma$ is trivial and the set of  bilinear forms $r$
on $V$ that are $\Gamma$-module maps satisfying $r \circ c_{V,V} = -r$. 
\end{proposition}

Next, we discuss invariant $2$-cocycles on $H$.

\begin{proposition}
\label{restriction is in YD}
Let $\sigma$ be a  $2$-cocycle on $H$ whose restriction on $\Gamma$ is trivial and let
$\alpha$ be defined as in \eqref{alpha}.
Then $\sigma$ is invariant if and only if $\alpha(x_i,\,x_j) =0$ whenever $g_ig_j\neq 1$,
i.e., if and only if $\alpha$ is morphism in $\YD$.
\end{proposition}
\begin{proof}
Suppose that $\sigma$ is invariant.
Taking $x=x_i,\, y=g\in G$ in \eqref{invariant 2cocycle} we obtain $\sigma(x_i,\, g)=1$
for all $i=1,\dots,n$ and $g\in G$. Similarly, taking $x=g$ and $y=x_i$ we get $\sigma(g,\,x_i)=1$.
Next, taking $x=x_i,\, y=x_j$ in \eqref{invariant 2cocycle} we obtain $\sigma(x_i,\, x_j)(1-g_ig_j)=0$
for all $i,j=1,\dots,n$. This means that $\sigma$ is a $\Gamma$-comodule map.
Hence, $\alpha$ is a $\Gamma$-comodule map.
Combining this with  Proposition~\ref{all twists on H} we see that $\alpha$ is a morphism in $\YD$.

Conversely, suppose that a $2$-cocycle $\sigma$ on $H$ is such that $\sigma|_{\Gamma \times \Gamma}=1$
and $\alpha$ is a morphism in $\YD$.  Then the  multiplication in the corresponding twisted Hopf algebra  $H^\sigma$
satisfies relations $g \cdot_\sigma x_i  =\chi_i(g) x_i \cdot_\sigma g$ and
\[
x_i \cdot_\sigma x_j - \chi_j(g_i) x_j \cdot_\sigma x_i = \alpha_{ij} (1-g_ig_j), \quad i,j=1,\dots,n.
\]
But the right hand side of the last equality is equal to $0$, so that  $H^\sigma = H$,
i.e.,  $\sigma$ is invariant.
\end{proof}

\begin{corollary}
\label{trivial restriction}
The map $\sigma \mapsto \left( \sigma|_{V\ot V} \right)_{sym} $ is an isomorphism between
the space of gauge equivalence classes of invariant $2$-cocycles on $H$ whose restriction on $\Gamma$ is trivial
and the space $\Sym_{\YD}^2(V^*)$.
\end{corollary}
\begin{proof}
Note that $\id_{V\ot V} - c_{V,V}$ is an invertible morphism, so by Propositions~\ref{all twists on H} and~\ref{restriction is in YD}
for an invariant twist $\sigma$ its restriction $\sigma|_{V\ot V}$ is a morphism in $\YD$. 
Using Lemma~\ref{symmetry} we get
\[
\sigma|_{V\ot V} \circ\tau_{V, V} =-\sigma |_{V\ot V} \circ c_{V,V} =\sigma,
\]
i.e., $\sigma|_{V\ot V}  \circ (\id_{V\ot V} -c_{V\ot V})=2\, \left( \sigma|_{V\ot V} \right)_{sym} $.
\end{proof}

We now analyze the general situation when $\sigma|_{\Gamma\times \Gamma}$ is not
necessarily trivial.

Let $\Gamma_0$ denote the subgroup of $\Gamma$ generated  by $g_i,\,i=1,\dots,n$.

\begin{proposition}
\label{Gamma-by-Gamma0}
Let $\sigma$ be an invariant $2$-cocycle on $H$. There is  $\rho\in Z^2(\Gamma/\Gamma_0,\, k^\times)$
such that $\sigma|_{\Gamma\times\Gamma}$ is cohomologous to $\rho\circ(\pi_{\Gamma_0}\times \pi_{\Gamma_0})$,
where $\pi_{\Gamma_0}: \Gamma\to \Gamma/\Gamma_0$ is the quotient homomomorphism.
\end{proposition}
\begin{proof}
It suffices to check that the alternating bilinear form $\text{alt}(\sigma):\Gamma\times\Gamma\to k^\times$ given by
\begin{equation}
\label{alt}
\text{alt}(\sigma)(g,\,h) = \frac{\sigma(g,\,h)}{\sigma(h,\,g)},\qquad g,h\in \Gamma
\end{equation}
vanishes on $\Gamma\times\Gamma_0$. But this follows  from
invariance of $\sigma$ since we must have $\sigma(g_i,\,g)
=\sigma(g,\,g_i) =1$ for all $i=1,\dots,n$ and $g\in G$.
\end{proof}

\begin{proposition}
\label{Hinv computed}
$\Hinv(H) \cong H^2( \Gamma/\Gamma_0,\,k^\times) \times \Alt^2_{\YD}(V^*)$.
\end{proposition}
\begin{proof}
By Corollary~\ref{trivial restriction} the group $\Alt^2_{\YD}(V^*)$ is identified with the normal
subgroup of $\Hinv(H)$ consisting
of gauge equivalence classes of invariant $2$-cocycles  with trivial restriction on $\Gamma$.

Next, there is a surjective Hopf algebra homomorphism $p: H \to k[ \Gamma/\Gamma_0]$
obtained by composing the canonical projection $H \to k[ \Gamma]$ with $\pi_{\Gamma_0}: \Gamma\to \Gamma/\Gamma_0$.
Thus, for any $2$-cocycle $\rho\in Z^2(\Gamma/\Gamma_0,\, k^\times)$ its pullback $p^*(\rho)$ is a $2$-cocycle on $H$.

Using the explicit formula \eqref{delta xP} for the comultiplication on $H$ we check that this $2$-cocycle
satisfies
\[
p^*(\rho)(x\1,\, y\1) x\2\ot y\2 = x\1\ot y\1  p^*(\rho)(x\2,\, y\2),\qquad x,y \in H.
\]
Indeed, for $x=hx_P,\, y=fx_Q$, where $h,f\in \Gamma$, both sides of this equality are equal to
$\rho(\pi_{\Gamma_0}(h),\, \pi_{\Gamma_0}(f)) hx_P \ot f x_Q$.

In particular, $p^*(\rho)$  is an invariant $2$-cocycle on $H$ and belongs to the center of $\Hinv(H)$.
Thus, there is a central embedding $H^2( \Gamma/\Gamma_0,\,k^\times)  \subset \Hinv(H)$.
The statement follows from Proposition~\ref{Gamma-by-Gamma0}.
\end{proof}

%%%%%%%%%%%%%%%%%%%%%%%%%%%%%%%%%%%%%%%%%%%%%%%%%%%%%%
%%%%%%%%%%%%%   CO_QT STRUCTURES  %%%%%%%%%%%%%%%%%%%%
%%%%%%%%%%%%%%%%%%%%%%%%%%%%%%%%%%%%%%%%%%%%%%%%%%%%%%
\section{Classification of co-quasitriangular structures on pointed Hopf algebras (Proof of Theorem~\ref{Thm 1})}
\label{Class of CQT}

%Given a bicharacter $r_0:\Gamma\times\Gamma\to k^\times$  there is
%a braiding on the  fusion category $\Vec_\Gamma$  given by
%\[
%c_{g,h} =r_0(g,\,h)\,\id_{gh}: g\ot h \to h \ot g,\qquad g,h\in
%\Gamma.
%\]
%Let $q:\Gamma\to k^\times$ be the quadratic form defined by $q(x) =r_0(g,\,g),\, g\in \Gamma$.
%Below we will identify pointed braided fusion category $\C(\Gamma,\, q)$
%with $\Corep(k\Gamma,\, r_0)$.

Let $\Gamma$ be a finite abelian group with a bicharacter $r_0:\Gamma\times\Gamma\to k^\times$.
Recall that $\C(\Gamma,\, r_0)$ denotes the fusion category of corepresentations of cosemisimple
co-quasitriangular Hopf algebra $(k[\Gamma],\, r_0)$. 

Let  $H$ be a co-quasitriangular  pointed Hopf algebra with the $r$-form $r:H\ot H\to k$
and  let  $\Gamma$ be the group of its group-likes. Clearly, $\Gamma$ must be abelian.
By the result of Angiono \cite{An13}, the algebra $H$ is generated  by $\Gamma$  and skew-primitive elements
$x_i,\, i=1,\dots,n,$ with $\Delta(x_i)=g_i\ot x_i + x_i\ot 1$ for some $g_i\in \Gamma,\, i=1,\dots,n$. 
Furthermore, there exist characters $\chi_i\in \widehat{\Gamma}$ such that $gx_ig^{-1}= \chi_i(g)x_i$.

Let $V$ be the span of $\{x_i\}_{i=1}^n$. Then $V$ is an object in $\YD$ with the above conjugation action
and coaction $\rho(x_i) =g_i\ot x_i$.

It is clear that $r$ restricts to a bicharacter $r_{0}$ on
$\Gamma$.  Observe that \eqref{cqt1} for the pair $(x, y) =
(x_{i}, g)$ yields $r_0 (g, g_{i}) = \chi_{i}^{-1} (g)$, and the
same condition for the pair $(x, y) = (g, x_{i})$ yields $r_0
(g_{i}, g) = \chi_{i} (g)$. Thus,
\[
r_0 (g_{i}, -) = \chi_{i} = r_0(-, g_{i})^{-1},\quad\mbox{for all } i=1, \dots, n.
\]
This condition  is equivalent to $V \in \Z_{sym}(\C(\Gamma,\, r_0)) \subset \YD$.
%(recall that for a
%Yetter-Drinfeld module $U$ we have $c_{U, V} (u \ot x_{i}) =
%\chi_{i}(g) x_{i} \ot u$, for a homogeneous element $u$ of degree
%$g$).  
It follows that $\chi_i(g_j) \chi_j(g_i) =1$  for all $i,j=1,\dots,n$.
Furthermore,  $\chi_{i}(g_{i}) = - 1$ for all $i=1,\dots,n$ (indeed, if $\chi_{i}(g_{i}) =1$
for some $i$
then the Hopf subalgebra of $H$ generated by $g_i$ and $x_i$ is non-semisimple commutative,
and, hence, infinite dimensional). 

Thus,
\[
V \in \Z_{sym}(\C(\Gamma,\, r_0))_- \subset \YD,
\]
where the embedding $\C(\Gamma,\, r_0) \hookrightarrow \YD$ is given by \eqref{hook}.
So $V$ is a quantum linear space of symmetric type.

\begin{remark}
\label{-1 on support V}
It follows from the classification of symmetric fusion categories \cite{De02} that 
in this situation there is
a character $\chi:\Gamma\to k^\times$ such that $\chi(g_i)=-1$ for all $i=1,\dots,n$
(in other words, $\chi =-1$ on the support of $V$). 
\end{remark}

It follows from results of Andruskiewitsch and Schneider \cite{AS98} and  Masuoka \cite{Ma01}
that  $H$ is a 2-cocycle twisting
of $\mathfrak{B} (V) \# k [\Gamma]$.  Since we are interested in the tensor
category of corepresentations of $H$ (which does not change under
2-cocycle twisting) from now on we will assume that $H=
\mathfrak{B} (V) \# k [\Gamma]$.

If $g \in \Gamma$ and $i \in \{1, \dots, n\}$, then it follows 
from \eqref{cqt2} that $r (g, x_{i}^{m}) = r (g,
x_{i})^{m}$, for all $m \geq 1$. Since $x_{i}$ is nilpotent, we
have $r (g, x_{i}) = 0$. Similarly, using \eqref{cqt3}, we deduce
that $r (x_{i}, g) = 0$.

For a non-empty subset $P$ of $\{1, \dots, n\}$ and for $g$, $h
\in G$ one can show, using again \eqref{cqt2} and \eqref{cqt3},
that $r (gx_{P}, h) = 0$ and $r (h, gx_{P}) = 0$ (the notation $x_P$
was introduced in Remark~\ref{xP remark}).

Consider now subsets $P$ and $Q$  of $\{1, \dots, n\}$. Condition
\eqref{cqt1} for the pair $(x, y) = (x_{P}, x_{Q})$ is
$$
(x_{P})_{(1)} (x_{Q})_{(1)} r \big( (x_{Q})_{(2)}, (x_{P})_{(2)} \big) =
r \big( (x_{Q})_{(1)}, (x_{P})_{(1)} \big) (x_{Q})_{(2)} (x_{P})_{(2)}.
$$
The terms in the coradical of $H$ of each side of the equality are
$g_{P} g_{Q} r(x_{Q}, x_{P})$ and $r(x_{Q}, x_{P})1$,
respectively. Since these have to be equal, we have
\begin{equation}
\label{gPgQ}
r (x_{Q}, x_{P}) (1 - g_{P}g_{Q}) = 0.
\end{equation}
This means that the restriction $r_{1}: = r|_{V \ot V}$ is a morphism of $\Gamma$-comodules,
i.e, a morphism in $\C(\Gamma,\, r_0) \subset \YD$. 
% Here I commented off the condition that $r_1$ is a comodule map,
% it is automatic since $\C(\Gamma,\, r_0)$ is embedded in $\YD$.
%
%Also, condition \eqref{cqt2} for the triple $(x, y, z) = (x_{Q},
%x_{P}, g)$ gives
%\[
%r (x_{Q}, x_{P}g) = \chi_{Q}(g) r (x_{Q}, x_{P})
%\]
%and the same condition for the triple $(x, y, z) =
%(x_{Q}, g, x_{P})$ gives
%\[
%r (x_{Q}, g x_{P}) = r (x_{Q}, x_{P}).
%\]
%Combining the two relations, we obtain
%$$
%r (x_{Q}, x_{P}) =  r (x_{Q}, g x_{P}) = \chi_{P} (g) r (x_{Q},
%x_{P}g) = (\chi_{P} \chi_{Q} ) (g) r (x_{Q}, x_{P}).
%$$
%Therefore,
%\begin{equation}
%\label{xPxQ}
%r (x_{Q}, x_{P}) (1 - \chi_{P} \chi_{Q}) = 0
%\end{equation}
%
%Observe that the restriction $r_{1}: = r|_{V \ot V}$ is a
%morphism in $_{\Gamma} ^{\Gamma} \mathcal{YD}$ if and only if
%$r_{1} (x_{i}, x_{j}) (1 - \chi_{i} \chi_{j}) = 0$ and $r_{1}
%(x_{i}, x_{j}) (1 - g_{i}g_{j}) = 0$, for all $i$, $j \in \{1,
%\dots, n\}$. Hence, equations \eqref{gPgQ} and \eqref{xPxQ} imply that
%$r_1$ is a morphism in $\YD$.

Thus, we showed that if $H$ admits a co-quasitriangular
structure $r$, then $V$ is a quantum linear space of symmetric
type and the pair $(r_0,\, r_1):= (r|_{\Gamma \times \Gamma}, r|_{V \otimes V})$
satisfies the conditions of Theorem~\ref{Thm 1}. 

It remains to prove that, conversely, given an object $V\in \Z_{sym}(\C(\Gamma,\, r_0)_-$
and a  pair $(r_{0}:\Gamma\times\Gamma\to k^\times,\, r_{1}:V\ot V \to k)$ 
there is  a unique $r$-form on $H$ such that $(r|_{\Gamma \times \Gamma}, r|_{V \otimes V}) =
(r_{0}, r_{1})$.

We need the following general facts.

\begin{lemma}
\label{cqt1 for H1}
Let $H = \mathfrak{B} (V) \# k [\Gamma]$. If $r : H \ot H \to k$
is convolution invertible linear map satisfying conditions
\eqref{cqt2} and \eqref{cqt3} then $r$ is a co-quasitriangular
structure on $H$ if and only if condition \eqref{cqt1} holds for
all pairs $(x, y) \in H_{1} \times H_{1}$, where $H_{1}$ is the
second term in the coradical filtration of $H$.
\end{lemma}
\begin{proof}
We need only prove sufficiency. By induction on $m$ we show that
condition \eqref{cqt1} holds for all pairs $(x, y)$ for which
either $x$ or $y$ is in $H_m$, the $m$-th term of the coradical
filtration.

Assume first that $x \in H_{1}$. Using induction on $k \geq 1$ we
show that condition \eqref{cqt1} holds for all pairs $(x, y)$ and
$(y, x)$ with $y \in H_{k}$. If $k = 1$ there is nothing to prove.
Assume that the claim is true for $k \geq 1$ and consider $z \in
H_{1}$. Then, using the induction hypothesis and \eqref{cqt2}, we
have
\begin{align*}
x_{(1)} (yz)_{(1)} r \big( (yz)_{(2)}, x_{(2)} \big) & = x_{(1)}
y_{(1)} z_{(1)} r (y_{(2)}z_{(2)}, x_{(2)}) \\
& =   x_{(1)} y_{(1)} z_{(1)}  r (y_{(2)},
x_{(2)}) r (z_{(2)},x_{(3)})\\
& = y_{(2)}  x_{(2)} z_{(1)} r (y_{(1)}, x_{(1)}) r
(z_{(2)}, x_{(3)})\\
& = y_{(2)} z_{(2)} x_{(3)} r (y_{(1)}, x_{(1)}) r (z_{(1)},
x_{(2)}) \\
& = y_{(2)} z_{(2)} x_{(2)} r (y_{(1)}z_{(1)}, x_{(1)}) \\
& = (yz)_{(2)} x_{(2)} r \big((yz)_{(1)},  x_{(1)}\big)
\end{align*}
and using \eqref{cqt3} we have
\begin{align*}
(yz)_{(1)} x_{(1)} r \big( x_{(2)}, (yz)_{(2)}\big)  & = y_{(1)}
z_{(1)} x_{(1)} r (x_{(2)}, y_{(2)}z_{(2)}) \\
& = y_{(1)} z_{(1)} x_{(1)} r (x_{(2)},
z_{(2)}) r (x_{(3)}, y_{(2)} ) \\
& = y_{(1)} x_{(2)} z_{(2)}  r (x_{(1)}, z_{(1)})
r (x_{(3)}, y_{(2)}) \\
& = x_{(3)} y_{(2)} z_{(2)} r (x_{(1)}, z_{(1)})  r (x_{(2)}, y_{(1)} ) \\
& = x_{(2)} y_{(2)} z_{(2)} r (x_{(1)}, y_{(1)}z_{(1)})\\
& = x_{(2)} (yz)_{(2)} r \big( x_{(1)}, (yz)_{(1)} \big).
\end{align*}
Since $H_{k+1} = H_{k} H_{1}$ it follows that \eqref{cqt1} holds
for all pairs $(x, y)$ and $(y,x)$ with $y \in H_{k+1}$. Thus,
\eqref{cqt1} holds for all pairs $(x, y)$ with either $x$ or $y$
in $H_{1}$.

Suppose now that \eqref{cqt1} holds for all pairs $(x, y)$ with
either $x$ or $y$ in $H_{m}$. Then a similar argument as the one
before shows that \eqref{cqt1} holds for all pairs $(x, yz)$ and
$(yz, x)$ with $y \in H_{m}$, $z \in H_{1}$ and arbitrary $x$.
Since $H_{m+1} = H_{m}H_{1}$, it follows that \eqref{cqt1} is
satisfied for all pairs $(x, y)$ with either $x$ or $y$ in
$H_{m+1}$. This proves the induction step and concludes the proof.
\end{proof}

\begin{lemma} \label{uniqueness}
Let $H$ be a Hopf algebra generated as an algebra by $h_{1},
\dots, h_{n}$ and such that the vector space $V$ spanned by
$h_{1}, \dots, h_{n}$ is a subcoalgebra. Then any
co-quasitriangular structure on $H$ is uniquely determined by its
restriction to $V \ot V$.
\end{lemma}

\begin{proof}
Suppose $r'$ and $r''$ are two co-quasitriangular structures on
$H$ such that $r' (h_{i}, h_{j}) = r'' (h_{i}, h_{j})$, for all
$i$, $j \in \{1, \dots, n\}$. If $i$, $i_{1}$, $\dots$, $i_{t} \in
\{1, \dots, n\}$ then, using \eqref{cqt3}, we have
\begin{align*}
r' (h_{i_{1}}  \cdots  h_{i_{t}}, h_{i}) & = r' \big( h_{i_{1}}, (h_{i})_{(1)} \big) \cdots r' \big( h_{i_{t}}, (h_{i})_{(t)} \big)\\
& = r'' \big( h_{i_{1}}, (h_{i})_{(1)} \big)  \cdots r'' \big( h_{i_{t}}, (h_{i})_{(t)} \big)\\
& = r'' (h_{i_{1}} \cdots h_{i_{t}}, h_{i})
\end{align*}
Thus, $r' (h, h_{i}) = r'' (h, h_{i})$, for all $h \in H$ and $i
\in \{1, \dots, n\}$.  Let $h \in H$ and $i_{1},\dots,i_{t}$ be 
in $\{1, \dots, n\}$. Then, using \eqref{cqt2}, we have
\begin{align*}
r' (h, h_{i_{1}}  \cdots  h_{i_{t}}) & = r' \big( h_{(1)}, h_{i_{t}} \big) \cdots r' \big( h_{(t)}, h_{i_{1}} \big)\\
& = r'' \big( h_{(1)}, h_{i_{t}} \big) \cdots r'' \big( h_{(t)}, h_{i_{1}} \big)\\
& = r'' (h, h_{i_{1}}  \cdots  h_{i_{t}})
\end{align*}
Since $h_{1}, \dots, h_{n}$ generate $H$ as an algebra, we
conclude that $r' = r''$.
\end{proof}

We now proceed to complete the proof of Theorem~\ref{Thm 1}. 

For $g \in \Gamma$ and $i = 1, \dots, n$ let $\gamma_{g}, \,
\xi_{i} \in H^{*}$ be defined by
$$
\gamma_{g} (hx_{P}) = \delta_{P, \emptyset} r_{0} (g, h), \qquad
\xi_{i} (hx_{P}) = \left\{ \begin{array}{ll} 0 & \textnormal{if } |P| \neq 1\\
r_{1} (x_{i}, x_{j}) & \textnormal{if } P = \{j\}
\end{array}\right.
$$
for all $h \in \Gamma$ and $P \subseteq \{1, \dots, n\}$. We will
show that $\varphi : H \to H^{* \textnormal{cop}}$ defined by
\[
\varphi(g) =
\gamma_{g} \mbox{ and } \varphi (x_{i}) = \xi_{i},\quad  \mbox{for all } g \in
\Gamma,\, i = 1, \dots, n, 
\]
is a bialgebra map. Using Remark~\ref{r-form = morphism}, this will prove that $r : H \ot H \to k$,
$r (x, y) = \varphi(x) (y)$, for all $x$, $y \in H$, is a linear
map, satisfying \eqref{cqt2} and \eqref{cqt3}, whose restriction
to $\Gamma \times \Gamma$, respectively $V \ot V$, is $r_{0}$,
respectively $r_{1}$.

Notice first that, since $r_{1} : V \ot V \to k$ is a morphism of
Yetter-Drinfeld modules, we have $\xi_{i} (x_{j}) (1 - \chi_{i}
\chi_{j}) = 0$ and $\xi_{i} (x_{j}) (1 - g_{i}g_{j}) = 0$, for all
$i$, $j \in \{1, \dots, n\}$. Using this, it is straightforward to
check that $\gamma_{g} \gamma_{h} = \gamma_{gh}$, $\gamma_{g}
\xi_{i} = \chi_{i} (g) \xi_{i} \gamma_{g}$, $\xi_{i} \xi_{j} =
\chi_{j} (g_{i}) \xi_{j} \xi_{i}$ and $\xi_{i}^{2} = 0$, for all
$g$, $h \in \Gamma$ and $i$, $j = 1, \dots, n$. For example,
$\xi_{i} \xi_{j} (hx_P) = 0 = \chi_{j} (g_{i}) \xi_j \xi_{i}
(hx_{P})$ when $|P| \neq 2$. If $1 \leq u < v \leq n$, then
\begin{align*}
\xi_{i} \xi_{j} (hx_{u}x_{v}) & = \xi_{i} (h g_{u} x_{v}) \xi_{j}
(hx_{u}) + \chi_{u}^{-1} (g_{v}) \xi_{i} (h g_{v} x_{u}) \xi_{j}
(hx_{v}) \\
& = \xi_{i} (x_{v}) \xi_{j} (x_{u}) + \chi_{i} (g_{j})^{-1}
\xi_{i} (x_{u}) \xi_{j} (x_{v})\\
& = \chi_{j}(g_{i}) \big( \xi_{i} (x_{u}) \xi_{j} (x_{v}) +
\chi_{j} (g_{i})^{-1}  \xi_{i} (x_{v}) \xi_{j} (x_{u}) \big)
\\
& = \chi_{j} (g_{i}) \big( \xi_{j} (h g_{u} x_{v}) \xi_{i} ( h
x_{u}) + \chi_{u}^{-1} (g_{v})  \xi_{j} (h g_{v} x_{u}) \xi_{i} (h
x_{v}) \big)\\
& = \chi_{j} (g_{i}) \xi_{j} \xi_{i} (h x_{u} x_{v}).
\end{align*}
Thus, $\varphi : H \to H^{* \textnormal{cop}}$, $\varphi(g) =
\gamma_{g}$ and $\varphi (x_{i}) = \xi_{i}$, for all $g \in
\Gamma$ and $i = 1, \dots, n$, is an algebra map. It is 
straightforward to check  that $\gamma_{g}$ is a
group-like element and $x_{i}$ a $(\gamma_{g_{i}},
\varepsilon)$-skew primitive element of $H^{* \textnormal{cop}}$
for every $g \in \Gamma$ and $i = 1, \dots, n$. Thus, $\varphi$ is
a bialgebra map. It follows that the linear map $r : H \ot H \to
k$, $r (x, y) = \varphi(x) (y)$, for all $x$, $y \in H$, satisfies
\eqref{cqt2} and \eqref{cqt3}. Moreover, $r|_{\Gamma \times
\Gamma} = r_{0}$ and $r|_{V \ot V} = r_{1}$. Since $r|_{k[\Gamma]
\ot k[\Gamma]}$ is convolution invertible, it follows that $r$ is
convolution invertible. Taking into account Lemma \ref{cqt1 for
H1}, in order to prove that $r$ is a co-quasitriangular structure
on $H$, we need only check that \eqref{cqt1} holds for every pair
of elements in the second term of the coradical filtration of $H$.
This is straightforward, as we next show for the pair $(gx_{i},
hx_{j})$:
\begin{eqnarray*}
r \big( (hx_{j})_{(1)}, (gx_{i})_{(1)} \big) (hx_{j})_{(2)}
(gx_{i})_{(2)} 
&=& r (hg_{j}, gg_{i}) hx_{j} gx_{i} + r (hx_{j}, gx_{i}) hg \\
&=& r (h, g) \chi_{i}^{-1} (h) gh x_{i} x_{j} + r (hx_{j}, gx_{i}) gh \\
&=&  gx_{i} hx_{j} r (h, g) + gg_{i} hg_{j} r (hx_{j}, gx_{i})\\
&=& (gx_{i})_{(1)} (hx_{j})_{(1)} r \big( (hx_{j})_{(2)}, (gx_{i})_{(2)} \big),
\end{eqnarray*}
where  for the third equality  we use the fact that $r (hx_{j},
gx_{i}) (1 - g_{i}g_{j}) = 0$. Thus, $r$ is a co-quasitriangular
structure on $H$ which restricts on $\Gamma$ to $r_{0}$ and on $V$
to $r_{1}$. The uniqueness of $r$ follows from Lemma~\ref{uniqueness}.
This completes the proof of Theorem~\ref{Thm 1}.

Let $\C(\Gamma,\, r_0,\, V,\, r_1)$ denote the braided tensor category constructed above.

\begin{example}
\label{E(n)}
Take $\Gamma=\mathbb{Z}/2\mathbb{Z}$ and  let $r_0$ be the unique non-trivial bicharacter of $\mathbb{Z}/2\mathbb{Z}$.
Let $V$ be a multiple of the non-identity simple object of $\Corep(k[\mathbb{Z}/2\mathbb{Z}])$. 
Then $V$ belongs to  $\Z_{sym}(\C(\mathbb{Z}/2\mathbb{Z}],\, r_0))$, so the Nichols Hopf algebra \cite{Ni78}
\begin{equation}
\label{E(V)}
E(V) :=\mathfrak{B}(V) \# k[\mathbb{Z}/2\mathbb{Z}]
\end{equation}
admits co-quasitriangular structures. According to Theorem~\ref{Thm 1} such structures are in bijection
with bilinear forms $r_1: V\ot V\to k$ or, equivalently, with $n$-by-$n$ square matrices, where $n=\dim_k(V)$.
This fact was first established in \cite{PvO99}.  
\end{example} 

%%%%%%%%%%%%%%%%%%%%%%%%%%%%%%%%%%%%%%%%%%%%%%%%%%%%%%
%%%%%%%%%%%%%   SYMMETRIC CENTER                           %%%%%%%%%%%%%%%%%%%%
%%%%%%%%%%%%%%%%%%%%%%%%%%%%%%%%%%%%%%%%%%%%%%%%%%%%%%
\section{The symmetric center of $\C(\Gamma,\, r_0,\, V,\, r_1)$}
\label{Sect:Zsym}

The symmetric center $\Z_{sym}(\C)$ of a braided tensor category $\C$ was defined
in Section~\ref{prelim:symcenter}.

Let $\Gamma,\, r_0,\, V,\, r_1$ be as in Theorem~\ref{Thm 1}.  Let $H =\mathfrak{B}(V) \# k[\Gamma]$
and $r: H\ot H \to k$ be the corresponding Hopf algebra and $r$-form. 

We have
\[
\Z_{sym}(\C(\Gamma,\, r_0, \, V,\, r_1)) \cong  \Corep(H_{sym},\, r|_{H_{sym} \ot H_{sym}}),
\]
where $H_{sym}$ is the Hopf subalgebra of $H$ defined by \eqref{Hsym formula}.

Let $b: \Gamma\times \Gamma \to k^\times$
be the symmetric bicharacter given by
\[
b(g,\,h) = r(g,\,h)r(h,\,g), \qquad g,h\in \Gamma.
\]

Let $\Gamma^{\perp}$ and $V^{\perp}$ denote the radicals of   $b$ and $(r_1)_{alt}$, respectively,  i.e.,
\begin{eqnarray*}
\Gamma^{\perp} & =& \{ g \in \Gamma \mid b (g, h) = 1 \textnormal{ for all } h \in \Gamma\},\\
V^{\perp} & = &\{ v \in V \mid (r_1)_{alt}(v, w) = 0 \textnormal{ for all } w \in V\}.
%V^{\perp} & = &\{ v \in V \mid r (v, w) = 0 \textnormal{ for all } w \in V\} =
%\{ v \in V \mid r (w, v) = 0 \textnormal{ for all } w \in V\},
\end{eqnarray*}
%where $(r_1)_{alt} = \frac{1}{2}(r_1 -r_1 \circ \tau_{V,V})$.

\begin{remark}
If $V\neq 0$ then $\Gamma^\perp \neq 0$ since $V\in \C(\Gamma^\perp,\, r_0|_{\Gamma^\perp\times \Gamma^\perp })$.
\end{remark}

\begin{lemma}
\label{centralizer lemma}
$H_{sym}$ is generated as an algebra by $\Gamma^{\perp}$ and $V^{\perp}$.
\end{lemma}
\begin{proof}
Since $H_{sym}$ is a Hopf subalgebra
of a pointed Hopf algebra with abelian coradical, it is pointed
with abelian coradical. By the result of Angiono \cite{An13}, $H_{sym}$ is generated by its
group-like and skew-primitive elements.

It is easy to see that an element $g \in \Gamma$ is in $H_{sym}$
if and only if $r (g, h) r (h, g) = 1$, for all $h \in \Gamma$.
Thus, the set of group-like elements of $H_{sym}$ is
$\Gamma^{\perp}$.

Let now $g \in \Gamma^{\perp}$ and let $x$ be a $(g, 1)$-primitive
element. If $g \notin \{ g_{i} \mid i = 1, \dots, n \}$ then $x$
is a scalar multiple of $1 - g$, so it is contained in $H_{sym}$.
If $g = g_{i}$ for some $i \in \{1, \dots, n\}$ then $x = a (1 -
g_{i}) + \sum_{j: g_{j} = g_{i}} a_{j}x_{j}$, for some $a$, $a_{j}
\in k$. Let $y = \sum_{j: g_{j} = g_{i}} a_{j}x_{j}$. We claim
that $x \in H_{sym}$ if and only if $y \in V^{\perp}$.
It is clear that $x \in H_{sym}$ if and only if $y \in
H_{sym}$. Next,
$$
y_{(1)} r (y_{(2)}, z_{(1)}) r (z_{(2)}, y_{(3)}) = \sum_{j: g_{j}
= g_{i}} a_{j} \Big( r (g_{j}, z_{(1)}) r (z_{(2)}, x_{j}) + r
(x_{j}, z) \Big) g_{i} + \varepsilon(z) y,\qquad z\in H, 
$$
so $y \in H_{sym}$ if and only if
\begin{equation} \label{condition}
\sum_{j: g_{j} = g_{i}} a_{j} \Big( r (g_{j}, z_{(1)}) r (z_{(2)},
x_{j}) + r (x_{j}, z) \Big) = 0
\end{equation}
for all $z \in \{hx_{l} \mid h \in \Gamma, l = 1, \dots, n\}$. For
$z = hx_{l}$, the left-hand side of \eqref{condition} becomes
\begin{align*}
\textnormal{LHS}\eqref{condition} & = \sum_{j: g_{j} = g_{i}}
a_{j} \Big( r (g_{j}, hg_{l}) r (hx_{l}, x_{j}) + r (x_{j}, hx_{l}) \Big) \\
& = \sum_{j: g_{j} = g_{i}} a_{j} \Big( r (g_{j}, g_{l}) r (x_{l},
x_{j}) + r (x_{j}, x_{l}) \Big)\\
& = r (g_{l}, g_{i}) r (x_{l}, y) + r (y, x_{l})\\
& = (r + r \circ c_{V, V}) (y, x_{l})  \\
& =  (r - r \circ \tau_{V, V}) (y, x_{l}) = 2 (r_1)_{alt}(y, x_{l}) ,
\end{align*}
where   we used the fact that 
$r|_{V\ot V}$ is a morphism in $\YD$ and Lemma~\ref{symmetry}.
Thus, $y \in H_{sym}$ if and only if $y
\in V^{\perp}$. It follows that  non-trivial skew-primitive elements of $H_{sym}$
generate  $V^\perp$, so the claim follows.
%$\Gamma^{\perp}$ and $V^{\perp}$
%generate all the $(g, 1)$-primitive elements of $H_{sym}$ for $g
%\in \Gamma^{\perp}$, and, since these primitive elements together
%with $\Gamma^{\perp}$ generate all the skew-primitive elements of
%$H_{sym}$, the claim follows.
\end{proof}

\begin{corollary}
$H_{sym} = \mathfrak{B} (V^\perp) \# k[\Gamma^{\perp}]$ and
\[
\Z_{sym}(\C(\Gamma,\, r_0, \, V,\, r_1))\cong 
 \C(\Gamma^\perp,\, r_0|_{\Gamma^\perp\times \Gamma^\perp},\, V^\perp,\, r_1|_{V^\perp\ot V^\perp}).
\]
\end{corollary}

\begin{corollary}
\label{properties}
\begin{enumerate}
\item[(i)] $\C(\Gamma,\, r_0, \, V,\, r_1)$ is symmetric if and only if $\C(\Gamma,\, r_0)$ is symmetric
and $r_1$ is symmetric.
\item[(ii)] $\Z_{sym}(\C(\Gamma,\, r_0, \, V,\, r_1))$ is semisimple if and only if $(r_1)_{alt}:V\ot V \to k$ is non-degenerate.
\item[(iii)] $\C(\Gamma,\, r_0, \, V,\, r_1)$ is factorizable if and only if $\C(\Gamma,\, r_0)$  is factorizable and $V=0$.
\end{enumerate}
\end{corollary}

\begin{example}
\label{cqt on E(V)}
Let $E(V)$ be the Hopf algebra from Example~\ref{E(n)} and let $r_1:V\ot V\to k$ be a bilinear form.
In this case  $c_{V,V} =-\tau_{V,V}$, 
so Corollary~\ref{properties}(i) says that the co-quasitriangular structure determined 
by $r_1$ is symmetric if an only if $r_1$ is symmetric (in the usual linear algebra sense). This was proved in  \cite{CC04}.
\end{example}

%%%%%%%%%%%%%%%%%%%%%%%%%%%%%%%%%%%%%%%%%%%%%%%%%%%%%%
%%%%%%%%%%%%%   Ribbon structures                                   %%%%%%%%%%%%%%%%%%%%
%%%%%%%%%%%%%%%%%%%%%%%%%%%%%%%%%%%%%%%%%%%%%%%%%%%%%%
\section{Ribbon structures on  $\C(\Gamma,\, r_0,\, V,\, r_1)$}
\label{Sect:ribbon}

Ribbon structures on braided tensor categories and ribbon elements of coquasitriangular Hopf algebras
were defined  in Section~\ref{prelim: ribbon}. We want  to classify ribbon structures of $\C (\Gamma,\, r_0,\, V,\, r_1)$. 
%For this end we need the following lemma.

\begin{lemma}
\label{existence of morphism}
Let $A$ be an abelian group and let $a_{1}, a_{2}, \dots, a_{n}$
be elements of $A$. Then there exists a group homomorphism $\gamma
: A \to \{ \pm 1\}$ such that $\gamma (a_{i}) = -1$, for all $i =
1, \dots, n$ if and only if there are no relations in $A$ of the
form $a_{i_{1}}a_{i_{2}} \cdots a_{i_{k}} = x^{2}$, with $k$ odd.
\end{lemma}
\begin{proof}
Homomorphisms $A \to \{ \pm 1\}$ are in bijection with homomorphisms $A/A^2 \to \{ \pm 1\}$.
Let $\bar{a}_1,\bar{a}_2\dots,\bar{a}_n$ be images of  $a_{1}, a_{2}, \dots, a_{n}$ in $A/A^2$.
If there are no relations in $A$ of the form
$a_{i_{1}}a_{i_{2}} \cdots a_{i_{k}} = x^{2}$ with $k$ odd then all relations
in $A/A^2$ (viewed as a $\mathbb{Z}/2\mathbb{Z}$-vector space)  
are of the form $\bar{a}_{i_{1}} + \bar{a}_{i_{2}} \cdots + \bar{a}_{i_{l}} = 0$ with $l$ even. 
Therefore, there is a well-defined homomorphism $A/A^2\to \{ \pm 1\}$ sending each $a_i$ to $-1$. 
The converse implication is trivial.
\end{proof}

\begin{proposition}
The set of ribbon structures on $\C (\Gamma,\, r_{0},\, V,\, r_{1})$ is non-empty
and is in bijection with 
the set of group homomorphisms $\gamma : \Gamma \to \{\pm
1\}$ such that $\gamma (g_{i}) = -1$, for all $i = 1, \dots, n$.
%The ribbon structure corresponding to $\gamma$ is
%$$
%\theta_{V} : V \to V, \quad v \mapsto \sum \gamma(v_{(1)}) \eta
%(v_{(2)}) v_{(0)}
%$$
\end{proposition}
\begin{proof}
Let $H = \mathfrak{B} (V) \# k[\Gamma]$. As explained in Section~\ref{prelim: ribbon}, 
ribbon structures on $\C (\Gamma,\, r_{0},\,
V,\, r_{1})$ are in bijection with group-like elements $\gamma \in
G(H^{*})$ satisfying $\gamma^{2} = (\eta \circ S) * \eta^{-1}$ and
$S^{2}_{H^{*}} (p) = \gamma^{-1}
* p * \gamma$, for all $p \in H^{*}$. Let $\gamma$ be such an
element. We have
$$
\gamma(g)^{2} = \gamma^{2} (g) = (\eta \circ S) * \eta^{-1} (g) =
r_{0} (g^{-1}, g) r_{0} (g, g) = 1
$$
for all $g \in \Gamma$, so $\gamma (\Gamma) \subseteq \{\pm 1\}$.
Now $S^{2}_{H^{*}} (p) = \gamma^{-1} * p * \gamma$ for all $p \in
H^{*}$, if and only if $S_{H}^{2} = \gamma^{-1} * \id_{H} *
\gamma$. Since both maps are algebra maps, they are equal if and
only if they agree on algebra generators. We have
\begin{alignat*}{4}
S_{H}^{2} (g)     &= g,          & \quad (\gamma^{-1} * \id_{H} * \gamma) (g)     &= g,\\
S_{H}^{2} (x_{i}) &= -x_{i},     & \quad (\gamma^{-1} * \id_{H} *
\gamma) (x_{i}) &=  \gamma^{-1} (g_{i}) x_{i}
\end{alignat*}
for all $g \in \Gamma$ and $i = 1, \dots, n$. Thus, $S_{H}^{2} =
\gamma^{-1} * \id_{H} * \gamma$ if and only if $\gamma (g_{i}) =
-1$ for all $i$.

It remains to show  that there always exists a homomorphism $\gamma :
\Gamma \to \{\pm 1\}$ such that $\gamma (g_{i}) = -1$, for all $i
= 1, \dots, n$. Since $k^{\times}$ is an injective
$\mathbb{Z}$-module, it is enough to show that there is such a homomorphism
on the subgroup $\Gamma_0 =\langle g_1,\dots g_n\rangle \subset \Gamma$.
Using Lemma \ref{existence of morphism}, we
have to show that there are no relations in $\Gamma_0$ of the form
$g_{i_{1}} g_{i_{2}} \cdots g_{i_{k}} = x^{2}$  with $k$ odd.
If $x = g_{i_{1}}^{e_{1}} g_{i_{2}}^{e_{2}} \cdots
g_{i_{t}}^{e_{t}}$ is an element of $\Gamma_0$ then
$$
r_{0} (x,\, x) = \prod_{r = 1}^{t} r_{0} (g_{i_{r}},
g_{i_{r}})^{e_{r}^{2}} \prod_{1 \leq r < s \leq t} \big( r_{0}
(g_{i_{r}}, g_{i_{s}}) r_{0} (g_{i_{s}}, g_{i_{r}})
\big)^{e_{r}e_{s}} = (-1)^{\sum_{r=1}^t\,e_r^2}. 
$$
In particular, $r_0(x,\,x)^2 =1$ for all $x\in \Gamma_0$. 
On the other hand, if $g_{i_{1}} g_{i_{2}} \cdots g_{i_{k}} = x^{2}$ with $k$
odd, then $r_0(x,\,x)^4 =r_{0} (x^{2}, x^{2}) = (-1)^{k} = -1$, which is a contradiction.
\end{proof}

%%%%%%%%%%%%%%%%%%%%%%%%%%%%%%%%%%%%%%%%%%%%%%%%%%%%%%
%%%%%%%%%%%%%   PARAMETERIZATION   %%%%%%%%%%%%%%%%%%%%
%%%%%%%%%%%%%%%%%%%%%%%%%%%%%%%%%%%%%%%%%%%%%%%%%%%%%%
\section{Metric quadruples (Proof of Theorem~\ref{Thm 2})}
\label{Metric 4sect}

It is well known that the $1$-categorical truncation of the $2$-category of pointed braided fusion
categories is equivalent to the $2$-category of {\em pre-metric
groups}. The objects of the former are pointed braided fusion categories and morphisms
are natural isomorphism classes of braided tensor functors.
The objects of the latter are pairs $(A,\, q)$, where $A$
is an abelian group and $q: A\to k^\times$ is a quadratic form, and
morphisms are orthogonal homomorphisms (see \cite{JS93} or \cite[Section 8.4]{EGNO15}).

The goal of this section is to extend this result to  tensor
categories that are not necessarily semisimple.

%%%%%%%%%%%%%%%%%%%%%%%%%%%%%%%%%%%%%%%%%%%%%%%%%%%%%%
\subsection{The category of metric quadruples}

Recall that a {\em groupoid} is a category in which all morphisms
are isomorphisms.

Let $\mathcal{P}$ denote  a groupoid  whose objects are pointed
braided tensor categories admitting a fiber functor and morphisms
are natural isomorphism classes of equivalences of braided tensor categories. Thus, objects of
$\mathcal{P}$ can be identified with  the co-representation
categories of co-quasitriangular pointed Hopf algebras.

Define a groupoid $\mathcal{Q}$ as follows. The  objects of
$\mathcal{Q}$ are quadruples $(\Gamma,\, q,\, V,\, r)$, where
$\Gamma$ is a finite  abelian group, $q\in \Quad_d(\Gamma)$ is a
diagonalizable quadratic form on $\Gamma$,  $V$ is an object in $\Z_{sym}(\C(\Gamma,\, q))_-$,
and $r:V\ot V \to k$ is an alternating bilinear form in $\C(\Gamma,\, q)$.
The set of morphisms
\[
\Hom_{\mathcal{Q}}((\Gamma,\, q,\, V,\, r),\, (\Gamma',\, q',\, V',\, r')) 
% =  \{ (\alpha,\, f) \mid \alpha\in O(\Gamma, \Gamma'),\, f\in \Iso_m(\ind_\alpha(V ),\, V') \} / \tilde,
\]
%where $O(\Gamma, \Gamma')$ is the set of orthogonal group  isomorphisms
%(i.e., isomorphisms $\alpha: \Gamma\to \Gamma'$ such that $q'=q\circ \alpha$),
%$\Iso_m(\ind_\alpha(V ),\, V')$ is the set of metric isomorphisms (i.e, 
is the set of  pairs $(\alpha,\, f)$ modulo certain equivalence relation $\sim$ which we describe next. 
Namely,   $\alpha: \Gamma\to \Gamma'$ is an orthogonal
group isomorphism (i.e., such that $q'\circ \alpha=q$)
and $f: \ind_\alpha(V )\to V'$ is an isomorphism
in $\C(\Gamma', q')$ such that  $r'\circ (f \ot f) = \ind_\alpha(r)$. Here
\[
\ind_\alpha: \C(\Gamma,\, q)\to \C(\Gamma',\, q')
\]
denotes the braided tensor equivalence induced by $\alpha$. Finally,  the relation $\sim$ identifies pairs
$(\alpha,\, f)$ and $(\alpha,\, -f)$.

\begin{definition}
We will call $\mathcal{Q}$ the groupoid of {\em metric quadruples}.
\end{definition}

Define a functor
\begin{equation}
\label{F} 
F:\mathcal{Q}\to \mathcal{P}
\end{equation}
as follows.  Given a quadruple $(\Gamma,\, q,\, V,\, r)$ as above
let $r_0:\Gamma\times\Gamma\to k^\times$ be a bicharacter such
that $q(g)=r_0(g,\,g)$ for all $g\in \Gamma$.  Set
\begin{equation}
\label{CGqVr} 
F(\Gamma,\, q,\, V,\, r) =\C(\Gamma,\, r_0,\, V,\, r). 
\end{equation}
A morphism $(\alpha,\, f)$ between  $(\Gamma,\, q,\, V,\, r)$ and
$(\Gamma',\, q',\, V',\, r')$ gives rise to  an isomorphism of
Hopf algebras $\varphi_{(\alpha,f)}: \mathfrak{B}(V) \# k[\Gamma]
\xrightarrow{\sim} \mathfrak{B}(V') \# k[\Gamma']$  given by
\begin{equation}
\label{varphi}
\varphi_{(\alpha,f)}(g) =\alpha(g),\quad \varphi_{(\alpha,f)}(x) = f(x)
\end{equation}
for all $g\in \Gamma$ and $x\in V$. If $r_0': \Gamma'\times \Gamma'\to k^\times$
is a bicharacter such that $q'(g) = r_0'(g,\,g),\, g\in\Gamma'$ then the bicharacter
$r_0'\circ(\alpha\times\alpha)/r_0$ is alternating and so is equal to  $\mbox{alt}(\mu)$
some $\mu\in Z^2(\Gamma/\Gamma_0,\, k^\times)$, see \eqref{alt}. This means that  the $r$-form
$r_0'\circ(\alpha\times\alpha)$ is a twisting deformation of  $r_0$ by means of
the above $2$-cocycle $\mu\in H^2(\Gamma/\Gamma_0,\, k^\times)$. But such $\mu$ defines an
invariant $2$-cocycle on $\mathfrak{B}(V) \# k[\Gamma]$ by Proposition~\ref{Hinv computed}.
Thus, $\varphi_{(\alpha,f)}$ gives rise to a well defined braided tensor  equivalence $F(\alpha,\,f)$
between $\C(\Gamma,\, r_0,\, V,\, r)$ and $\C(\Gamma',\, r_0',\, V',\, 'r)$.

%%%%%%%%%%%%%%%%%%%%%%%%%%%%%%%%%%%%%%%%%%%%%%%%%%%%%%
\subsection{Proof of Theorem~\ref{Thm 2}}
\label{proof section}

Our goal is to show that the functor $F$ defined in the previous Section is an equivalence. 

Let  $H := \mathfrak{B}(V) \# k[\Gamma]$ and $H' :=
\mathfrak{B}(V') \# k[\Gamma']$ with  co-quasitriangular structures on $H$ and
$H'$ defined by $(r_0,\, r_{1})$ and $(r_0',\, r'_{1})$.

\begin{lemma}
\label{lemma 1} 
Co-quasitriangular Hopf algebra isomorphisms $H
\to H'$ are in bijection with pairs $(\alpha,\, f)$, where
$\alpha: \Gamma\to \Gamma'$ is a group isomorphism,  $f:
\ind_\alpha(V )\to V'$ is an isomorphism such that  
$r_0'\circ(\alpha \times \alpha) = r_0$,
and $r_{1}'\circ (f \ot f) = \ind_\alpha(r_{1})$.
\end{lemma}
\begin{proof}
Let $f : H \to H'$ be an isomorphism of co-quasitriangular Hopf
algebras. Since $f$ takes group-like elements to group-like
elements, it restricts to a group isomorphism $\alpha: \Gamma \to
\Gamma'$. Let us show that $f$ induces an isomorphism
$\ind_{\alpha} (V) \to V'$.

Notice first that $\alpha (g_{i}) \in \{g'_{j}\}$. Indeed, if this
is not the case, then $f (x_{i}) = a (1 - f(g_{i}))$, for some $a
\in k$. But $f (x_{i})$ anti-commutes with $f (g_{i})$, while $a
(1 - f(g_{i}))$ commutes with $f(g_{i})$, so $f(x_{i}) = 0$. This,
however, contradicts the injectivity of $f$. Thus, $\alpha (g_{i})
\in \{g'_{j}\}$, for all $i$. It follows that there exist scalars
$a_{i}$, $b_{ij} \in k$ such that $b_{ji} \big( g'_{j} - f(g_{i})
\big) = 0$ and
$$
f (x_{i}) = a_{i} (1 - f(g_{i})) + \sum_{j} b_{ji} x'_{j},\qquad
i=1,\dots,n.
$$
We must have $f(g) f(x_{i}) = \chi_{i} (g) f (x_{i}) f(g)$, for
all $g \in \Gamma$. Now
\begin{align*}
f (g) f (x_{i}) & =  a_{i} (f(g) - f(gg_{i})) + \sum_{j} b_{ji} f(g) x'_{j} \\
& = a_{i} (f(g) - f(gg_{i})) + \sum_{j} b_{ji} \chi_{j}' \big( f
(g) \big) x'_{j} f (g)
\end{align*}
and
$$
\chi_{i} (g) f (x_{i}) f(g) = a_{i} \chi_{i} (g) (f(g) -
f(gg_{i})) + \sum_{j} b_{ji} \chi_{i} (g) x'_{j} f (g).
$$
Thus, $f(g) f(x_{i}) = \chi_{i} (g) f (x_{i}) f(g)$ if and only if
$a_{i} = a_{i} \chi_{i} (g)$ and $\chi_{j}' \big( f (g) \big)
b_{ji} = \chi_{i} (g)  b_{ji}$, for all $j$ and $g \in \Gamma$.
Taking $g = g_{i}$ in the first condition, we obtain $a_{i} = 0$.
The second condition is equivalent to $b_{ji} (\chi_{i} -
\chi'_{j} \circ f) = 0$. It follows that
$$
f (x_{i}) = \sum_{j} b_{ji} x'_{j}
$$
where $b_{ji} \big( g'_{j} - \alpha (g_{i}) \big) = 0$ and $b_{ji}
(\chi_{i} \alpha^{-1} - \chi'_{j}) = 0$. This means precisely that
the restriction of $f$ to $V$  is a morphism in $\C(\Gamma',\,r_0')$ from $\ind_{\alpha} (V)$ to $V'$. 
Since $f$ is an isomorphism, the restriction is also an isomorphism with
inverse the restriction of the inverse of $f$ to $V'$.

We have proven that if $f : H \to H'$ is a Hopf algebra
isomorphism then it induces, by restriction, isomorphisms $\alpha:
\Gamma \to \Gamma'$ and $f: \ind_{\alpha} (V) \to V'$. We have $r'
\circ (f \ot f) = r$ if and only if $r_0' \circ (\alpha \times
\alpha) = r_0$ and $r_{1}' \circ (f \ot f) = \ind_\alpha(r_{1})$.
It is easy to check that every such data comes from an isomorphism
of co-quastriangular Hopf algebras $H \to H'$. This completes the
proof.
\end{proof}

\begin{lemma}
\label{lemma 2} 
Let $\sigma$ be a $2$-cocycle on $H =
\mathfrak{B}(V) \# k[\Gamma]$ such that $\sigma|_{\Gamma\times
\Gamma}=1$. Then the twisted $r$-form $r^\sigma$ (see
\eqref{rsigma}) on $H^\sigma$ satisfies
\begin{equation}
r^\sigma|_{V\ot V}  = r|_{V\ot V} + 2 (\sigma|_{V\ot V})_{sym}.
\end{equation}
\end{lemma}
\begin{proof}
Using formula \eqref{rsigma}  we compute
\begin{eqnarray*}
r^\sigma (x_i,\, x_j)
&=& r(g_i,\, g_j) \, \sigma^{-1}(x_i,\, x_j) +  r(x_i,\, x_j) + \sigma(x_j,\, x_i) \\
&=& r(x_i,\, x_j)  + \left( \sigma(x_j,\, x_i) - \chi_i(g_j) \sigma(x_i,\, x_j)  \right)  \\
&=& r(x_i,\, x_j)  +(\sigma\circ \tau_{V,V} -\sigma\circ c_{V,V} \circ \tau_{V,V}) (x_i,\, x_j).
%&=& r(x_i,\, x_j)  + 2 (\sigma_{alt} \circ \tau) (x_i,\, x_j).
\end{eqnarray*}
Since restrictions of $r$ and $r^\sigma$ on $V\ot V$  are morphism in $\YD$ we  conclude, using  Lemma~\ref{symmetry}, that
\[
\sigma|_{V\ot V} \circ \tau_{V,V} -\sigma|_{V\ot V}  \circ c_{V,V} \circ \tau_{V,V}  = 2 (\sigma|_{V\ot V})_{sym},
\]
which implies the statement.
\end{proof}

We now proceed with the proof of Theorem~\ref{Thm 2}. We need to
show that the functor $F$ \eqref{F} is surjective and fully faithful.

(1) {\em $F$ is surjective}.

For this end it suffices to check that the co-quasitriangular
structure on $H$ defined using $r :V\ot V \to k$ is  gauge equivalent, by means of an invariant
$2$-cocycle on $H$, to the one defined using an alternating morphism $V\ot V \to k$ in $\C(\Gamma,\, r_0)$. 
Let $\sigma$ be an invariant $2$-cocycle on
$H$  such that $\sigma|_{\Gamma\times \Gamma} =1$ (such
$2$-cocycles are classified in Corollary~\ref{trivial
restriction}) and let  $r^\sigma = \sigma_{21} * r * \sigma^{-1}$.
By Lemma~\ref{lemma 2} we have
\[
r^\sigma|_{V\ot V} =r|_{V\ot V}  +2 (\sigma|_{V\ot V})_{sym}.
\]
Thus, we can take $\sigma$ such that $(\sigma|_{V\ot
V})_{sym} = -\frac{1}{2}(r\circ \tau)_{sym}$, so that $r^\sigma|_{V\ot V}   =(r|_{V\ot V} )_{alt}$, i.e., 
$r^\sigma|_{V\ot V}$ is alternating.

Since  $\C(\Gamma,\, r_0, \, V,\, r)$ and $\C(\Gamma,\, r_0, \, V,\, r^\sigma)$ are equivalent braided tensor
categories, the surjectivity of $F$  follows.

(2) {\em $F$ is faithful.} We need to check that $F$ is injective
on morphisms. It is clear from definitions that $F(\alpha, f) = F(\alpha',\, f')$
implies $\alpha = \alpha'$.  Therefore, it remains to check that for an automorphism
$(\id_\Gamma,\, f)$ of $(\Gamma,\, q,\, V,\, r)\in \mathcal{Q}$ one has $F(\id_\Gamma,\, f)\cong \id_{\C(\Gamma,\, q,\, V,\, r)}$ 
as a tensor functor if and only if  $f=\pm \id_V$. One implication is clear since $-\id_V$ preserves any bilinear form. 

Note that the Hopf algebra automorphism $\varphi_{(\id_\Gamma,f)}$ of $H$ defined in \eqref{varphi} gives rise
to a trivial tensor autoequivalence of $\Corep(H)$ if and only if  it is given by
\[
h \mapsto \chi \rightharpoonup h \leftharpoonup \chi^{-1},\quad h\in H
\]
for some character $\chi\in H^*$.  
The condition that it preserves $r$ is equivalent to $\chi(g_i)\chi(g_j)=1$ for all $i,j=1,\dots,n$, i.e., 
to $\chi$ being identically equal to $1$ or $-1$ on the support of $V$. By the Remark~\ref{-1 on support V}
there exists $\chi$ such that this value is $-1$, so that $f=\pm \id_V$. 

(3) {\em $F$ is full.} We need to check that $F$ is surjective on
morphisms. We claim that any braided tensor equivalence $\Phi$ between
$\Corep(\mathfrak{B}(V) \# k[\Gamma],\, r)$ and
$\Corep(\mathfrak{B}(V') \# k[\Gamma'],\, r')$ is isomorphic to one
coming from a co-quasitriangular Hopf algebra isomorphism $\mathfrak{B}(V) \#
k[\Gamma] \to \mathfrak{B}(V') \# k[\Gamma']$.   By the result of
Davydov \cite{Da10} $\Phi$  corresponds to a pair $(f,\,
\sigma)$, where $\sigma$ is a $2$-cocycle on $\mathfrak{B}(V) \#
k[\Gamma]$ and $f: (\mathfrak{B}(V) \# k[\Gamma])^\sigma \to
\mathfrak{B}(V') \# k[\Gamma']$ is Hopf algebra isomorphism such
that
\[
r \circ (f \otimes f) = r^\sigma.
\]
The last condition corresponds to the braided property of the
equivalence.

We must have $\sigma|_{\Gamma\times \Gamma} =1$ since non-trivial
twisting changes the braided equivalence class of
$\Corep(k[\Gamma],\, r_0)$. By Lemma~\ref{lemma 2} we have
\[
\left( r \circ (f \otimes f) - r   \right)
|_{V\ot V} = 2 (\sigma|_{V\ot V})_{sym}.
\]
The left hand side is of the above equality is alternating, while the right hand
side is symmetric.  Hence, both sides are equal to $0$
and so $\sigma$ is gauge equivalent to the trivial $2$-cocyle by
Proposition~\ref{all twists on H}. This means that $\Phi$ is isomorphic
to the equivalence induced by a  co-quasitriangular Hopf algebra isomorphism, 
so the result follows from Lemma~\ref{lemma 1}.

\begin{remark}
\label{r from Ext}
We can give a conceptual explanation of the reason why $(r_1)_{alt}$
is an invariant of the braided tensor category $\C:=\C(\Gamma,\, r_0,\, V,\, r_1)$.

Let $g\in \Gamma$. We will also use $g$ to denote the
corresponding invertible $H$-comodule. Recall that $\Ext_\C(g,\,
1) \cong  P_{1,g}(H)/k(1-g)$, where $P_{1,g}(H)$ denotes the space
of  $(1,g)$-skew primitive elements of $H$.  Explicitly, elements
of $\Ext_\C(g,\, 1)$ are in bijection  with equivalence classes of
short exact sequences
\[
0 \to 1 \xrightarrow{\iota} V_x \xrightarrow{p} g \to 0,
\]
where $1$ denotes the trivial comodule $k$.  The $2$-dimensional
comodule  $V_x$ is  a vector space with a basis $v_0,\, v_1$ and
$H$-coaction given by
\begin{equation}
\label{v0 and v1} \rho(v_0) = v_0 \ot 1,\qquad \rho(v_1) = v_0 \ot
x + v_1\ot g,
\end {equation}
where $x\in P_{1,g}(H)$.

Let $x' \in P_{1,g'}(H), \,g'\in \Gamma$, be another
skew-primitive element of $H$, let
\[
0 \to 1 \xrightarrow{\iota'} V_{x'} \xrightarrow{p'} g' \to 0
\]
be the corresponding extension, and let $v_0',\, v_1'$ be a basis
of $V_{x'}$ defined analogously to \eqref{v0 and v1}.

Let $\beta_{x,x'}= c_{V_{x'},V_x}\circ c_{V_x,V_{x'}}$ denote the
square of the braiding on $V_x\ot V_{x'}$.  Using formula
\eqref{braiding from r-form}   one computes
\begin{eqnarray*}
\beta_{x,x'}(v_0 \ot v_0') &=& v_0 \ot v_0', \\
\beta_{x,x'}(v_0 \ot v_1') &=& v_0 \ot v_1', \\
\beta_{x,x'}(v_1 \ot v_0') &=& v_1 \ot v_0', \\
\beta_{x,x'}(v_1 \ot v_1') &=& v_1 \ot v_1' + (r_1(x,\,x') +
r_0(g,\,g')r_1(x',\,x)) v_0\ot v_0'.
\end{eqnarray*}
Let $s:= r_1- r_1\circ \tau$.
Combining Lemma~\ref{symmetry} with  above computation  we see that
\begin{equation}
\label{betaxx'}
\beta_{x,x'}=  \id_{V_x\ot V_{x'}} + s(x,\,x') (p \ot
p')\circ (\iota\ot \iota')
\end{equation}
for all $x,x'\in V = \Ext_\C(\Gamma,\, 1)$ (note that $(p \ot p')\circ
(\iota\ot \iota')\in \End_\C(V_x\ot V_{x'})$ whenever
$s(x,\,x')\neq 0$).

It follows from \eqref{betaxx'} that $s= (r_1)_{alt}$ is an invariant of the
braided equivalence class of $\C$ (a computation
establishing  this fact  is straightforward and can be found in the proof of \cite[Proposition 6.7]{BN15}).
\end{remark}

Let $\C= \C(\Gamma,\, q,\, V,\, r)$.  Theorem~\ref{Thm 2} allows to compute the group $\Aut^{br}(\C)$
of isomorphism classes of  braided autoequivalences of $\C$. Namely, let
\begin{eqnarray*}
\Aut(V,\, r) &:=& \{ f\in \Aut_{\C(\Gamma,\,q)}(V) \mid r\circ (f\ot f) = r \} / \{\pm \id_V\}, \\
O(\Gamma,\, q,\,r) &:=& \{\alpha\in \Aut(\Gamma) \mid q\circ \alpha= q \mbox{  and }  \ind_\alpha(r),\,r \mbox{ are congruent in } \C(\Gamma,\,q)\}.
\end{eqnarray*}

\begin{corollary}
\label{group of autoequivalences}
There is a short exact sequence
\begin{equation}
\label{se for Autbr}
1 \to \Aut(V,\,r) \to \Aut^{br}(\C) \to O(\Gamma,\, q,\,r)\to 1.  
\end{equation}
\end{corollary}

\begin{example}
Let us consider $\C=\Corep(E(V),\,r)$, where $E(V)$ is the  Hopf algebra from  Example~\ref{E(n)} with the co-quasitriangular structure
given by the zero bilinear form on $V$ (this structure is symmetric). In this case $\Gamma =\mathbb{Z}/2\mathbb{Z}$,
so $O(\Gamma,\, q,\,r)=1$ and  Corollary~\ref{group of autoequivalences} implies  that
$\Aut^{br}(\C) = GL_n(V)/ \{\pm \id_V\}$, cf.\ \cite{BN15}. 
\end{example}

%%%%%%%%%%%%%%%%%%%%%%%%%%%%%%%%%%%%%%%%%%%%%%%%%%%%%%%%%%%%%%%%%%%%%%%%%%
%%%%%%%%%%%%%%%  THE DRINFELD CENTER                                                                       %%%%%%%%%%%%%%%%%%%%%%%
%%%%%%%%%%%%%%%%%%%%%%%%%%%%%%%%%%%%%%%%%%%%%%%%%%%%%%%%%%%%%%%%%%%%%%%%%%
\section{The Drinfeld center of  a pointed braided tensor  category}
\label{Sect: Drinfeld center}

It is well known that the Drinfeld center of a pointed braided fusion category is pointed (see, e.g., \cite[Proposition 5.8]{DN13}).
This is no longer true in the non-semisimple case. Indeed, the Drinfeld center is always factorizable, cf.\  Corollary~\ref{properties}(iii). 

Let   $\C=\C(\Gamma,\, r_0,\, V,\, r_1)$
be  a pointed braided tensor category corresponding to a  non-zero {\em self-dual} Yetter-Drinfeld module $V$.
In this Section we show that the trivial component of the universal grading of $\Z(\C)$ is pointed, i.e., $\Z(\C)$ is nilpotent
of nilpotency class $2$ (\cite{GN08}).

%%%%%%%%%%%%%%%%%%%%%%%%%%%%%%%%%%%%%%%%%%%%%%%
\subsection{The group-like elements of $D(H)^{*}$}
\label{adjoint subcategory}

It is well known \cite[Proposition 10]{R93} that for a finite dimensional Hopf algebra $H$
 the group of central group-like elements of $D(H)$ is
\[
G \big( D(H)^{*} \big)
\cong G \big( D(H) \big) \cap Z \big( D(H) \big).
\]
More explicitly,
\begin{align*}
G \big( D(H)^{*} \big) & = \{ g \ot \gamma \mid  (g, \gamma) \in G
(H) \times G (H^{*}), \\
& \qquad \qquad \qquad \qquad  \sum \gamma (h_{(1)}) h_{(2)} g =
\sum \gamma (h_{(2)}) g h_{(1)},\, \forall h \in H\}
\end{align*}
Let $V$ be a quantum linear space of symmetric type and let $H = \mathfrak{B} (V) \# k[\Gamma]$.
It is not hard to check that
$$
G \big( D(H)^{*} \big) = \{ g \ot \gamma \mid (g, \gamma) \in
\Gamma \times \widehat{\Gamma}, \, \gamma (g_{i}) = \chi_{i} (g), \mbox{for all } i =
1, \dots, n\}.
$$
Let $b : (\Gamma \times \widehat{\Gamma}) \times (\Gamma
\times \widehat{\Gamma}) \to k^{\times}$ be the canonical non-degenerate bicharacter
defined by
\[
b \big( (g, \chi), (g', \chi') \big) = \chi (g') \chi' (g), \qquad
(g, \chi), (g', \chi') \in \Gamma \times \widehat{\Gamma}.
\]
Consider the  subgroup
\begin{equation}
\label{Sigma}
\Sigma := \langle (g_{i}, \chi_{i}^{-1}) \mid i = 1, \dots, n \rangle \subset \Gamma \times \widehat{\Gamma}.
\end{equation}

\begin{remark}
The above $\Sigma$ is an isotropic subgroup of $\Gamma \times \widehat{\Gamma}$, i.e., $\Sigma\subset \Sigma^{\perp}$,
where $\Sigma^{\perp}$ is the orthogonal complement of $\Sigma$ with respect to $b$.
\end{remark}

It follows that
\begin{equation}
\label{group-likes of the dual double}
G \big( D(H)^{*} \big) \cong G \big( D(H) \big) \cap Z \big( D(H)
\big) \cong \Sigma^{\perp}.
\end{equation}

Let $H$ be a Hopf algebra. It was shown in \cite{GN08} that
for any Hopf subalgebra $K$ of $H$ contained in the center of $H$
tensor category $\Rep (H)$ is graded by $G (K^{*})$.
The trivial component of this grading is $\Rep
(H/HK^{+})$. The maximal central Hopf subalgebra of $H$ provides
the universal grading of $\Rep (H)$. In this case, the trivial
component is $\Rep (H)_{\textnormal{ad}}$, the adjoint subcategory
of $\Rep (H)$.

If $H$ is quasitriangular then the maximal central Hopf subalgebra of $H$
is the group algebra of $G(H) \cap Z (H)$ and the universal
grading group of $\Rep (H)$ is the group of characters of $G(H)
\cap Z (H)$.

We will be interested in the adjoint subcategory of
$\Z(\Corep(H))$, where $H = \mathfrak{B} (V) \# k[\Gamma]$ for a
quantum linear space $V\in \YD$ of symmetric type. We have
\begin{equation}
\label{ZCorepH}
\Z(\Corep (H)) \cong \Rep(D(H)^{\textnormal{cop}}) \cong \Corep(D(H)^{*\textnormal{op}}).
\end{equation}
Since $D(H)^{\textnormal{cop}}$ is quasitriangular, the universal
grading group of $\Z(\Corep (H))$ is isomorphic to
$\widehat{\Sigma^{\perp}}$, according to \eqref{group-likes of the
dual double}.

\begin{remark}
The Drinfeld double $D(\mathfrak{B} (V) \# k[\Gamma])$ was studied by Beattie in \cite{B03}.
\end{remark}

%%%%%%%%%%%%%%%%%%%%%%%%%%%%%%%%%%%%%%%%%%%%%%%%%%%%%%%%%%%%%%%%%%%%
\subsection{The adjoint subcategory of the center of $\C$ }

%Let $V \in \, ^{\Gamma}_{\Gamma} \mathcal{YD}$ be a quantum linear
%space of symmetric type associated to the datum $(g_{1}, \dots,
%g_{n}, \chi_{1}, \dots, \chi_{n})$, and let $D(V)$ be the Drinfeld
%double of $V$. Recall that $D(V) = W \oplus W^{*} \in \,
%^{\Sigma}_{\Sigma} \mathcal{YD}$, where $\Sigma$ is the subgroup
%of $\Gamma \times \widehat{\Gamma}$ generated by $(g_{i},
%\chi_{i}^{-1})$, $i = 1, \dots, n$, and $W$ is the quantum linear
%space associated to the datum $\big( (g_{1}, \chi_{1}^{-1}),
%\dots, (g_{n}, \chi_{n}^{-1}), \varphi_{1}, \dots, \varphi_{n}
%\big)$, with $\varphi_{i} : \Sigma \to k^{\times}$, $\varphi_{i}
%(g, \chi) = \chi_{i} (g)$, for all $i$.

If  $(H, r)$ is a
co-quasitriangular Hopf algebra then
$$
\iota_{r} : H \to D(H)^{* \, \textnormal{op}}, \quad \iota_{r} (x)
= x_{(1)} \ot r (-, x_{(2)})
$$
is a Hopf algebra homomorphism. This corresponds to the embedding $\C \hookrightarrow \Z(\C)$
from Remark~\ref{braided TC in center}.

Let $V \in \, ^{\Gamma}_{\Gamma} \mathcal{YD}$ be a quantum linear
space of symmetric type associated to the datum $(g_{1}, \dots,
g_{n}, \chi_{1}, \dots, \chi_{n})$ and let $H = \mathfrak{B} (V)
\# k[\Gamma]$. Suppose $H$ admits co-quasitriangular structures.
Then, according to Theorem~\ref{Thm 1}, $r$-forms on $H$ are
in bijection with pairs $(r_{0}, r_{1})$, where $r_{0}$ is a
bicharacter of $\Gamma$ such that $V\in \Z_{sym}(\C(\Gamma,\,r_0))_-$
and $r_{1} : V \ot V \to k$ is a morphism in  $\C(\Gamma,\,r_0)$.

Consider the $r$-form $r$ on $H$ corresponding to the pair
$(r_{0}, 0)$. Then  $\iota_{r} (g_{i}) = g_{i} \ot \chi_{i}^{-1}$
and $\iota_{r} (x_{i}) = x_{i} \ot \varepsilon$. Thus, $D(H)^{*}$
contains group-like elements $g_{i} \ot \chi_{i}^{-1}$ and
$(g_{i} \ot \chi_{i}^{-1}, 1)$-skew primitive elements $x_{i} \ot
\varepsilon$, $i=1,\dots,n$.

Assume now that $V$ is a self-dual object of $\YD$.
In this case the set $\{(g_{i},\, \chi_{i}) \mid i = 1, \dots, n\}$ is closed under taking inverses.
Let  $r_{1} : V \ot V \to k$  denote a non-degenerate evaluation morphism in $^{\Gamma}_{\Gamma} \mathcal{YD}$.
Let $r'$ be the $r$-form on $H$ corresponding to the pair $(r_{0},\, r_{1})$. Then
$\iota_{r'} (g_{i}) = g_{i} \ot \chi_{i}^{-1}$ and $\iota_{r'}
(x_{i}) = g_{i} \ot r' (-,\, x_{i}) + x_{i} \ot \varepsilon$. We see
that in this case $D(H)^{*}$ contains, in addition to the above, $(g_{i} \ot
\chi_{i}^{-1}, 1)$-skew primitive elements $g_{i} \ot r' (-,\, x_{i})$.

Let $A$ be the Hopf subalgebra of $D(H)^{*}$ generated by 
group-like elements $g_{i} \ot \chi_{i}^{-1}$ and 
skew-primitive elements $x_{i} \ot \varepsilon$ and $g_{i} \ot r'(-,\, x_{i})$, $i = 1, \dots, n$.

\begin{remark}
\label{dimA} By definition \eqref{Sigma}, the group of group-likes
of $A$ is $\Sigma$. The above skew-primitive elements  $x_{i} \ot
\varepsilon$ and $g_{i} \ot r'(-,\, x_{i}),\,i = 1, \dots, n$
constructed above are linearly independent and form a
$2n$-dimensional quantum linear space of symmetric type in
$^{\Sigma}_{\Sigma}\mathcal{YD}$. Therefore,
\[
\dim_k (A) = |\Sigma|\, 2^{2n}.
\]
\end{remark}

\begin{proposition}
\label{quotient_of_D(H)}
Let $K=k[G(D(H^*))]$.  We have
\begin{equation}
\label{A*}
A^{*} \cong D(H) / D(H)K^{+}.
\end{equation}
\end{proposition}
\begin{proof}
Let $\pi : D(H) \to D(H)/ D(H)K^{+}$ be the canonical projection.
We claim that $A$ is the image of the dual Hopf algebra homomorphism
$\pi^{*} : \big( D(H) / D(H)K^{+} \big)^{*} \to D (H)^{*}$.

According to (\ref{group-likes of the dual double}), $K$ is the
group algebra of $\{ \gamma \ot g \mid (g, \gamma) \in
\Sigma^{\perp} \}$. Let
\[
e = \frac{1}{|\Sigma^{\perp}|} \sum_{(g,\gamma) \in \Sigma^{\perp}} \gamma \ot g.
\]
Then $e$ is a central
idempotent of $D(H)$, $ze = \varepsilon (z) e$, for all $z \in K$,
and $K^{+} = (1 - e) K$. Thus, $D(H) K^{+} = D(H) K (1 - e) = (1 -
e) D(H)$. The image of $\pi^{*}$ is
$$
\textnormal{Im} (\pi^{*}) = \{ f \in \big( D(H) \big)^{*} \mid f
(z) = f (ez), \textnormal{ for all } z \in D(H) \}.
$$

It is easy to check that any $f \in \{ g_{i} \ot
\chi_{i}^{-1}, x_{i} \ot \varepsilon, g_{i} \ot r' (-, x_{i})\}$ satisfies
$f ( (\gamma \ot g)z) = f (z)$, for all $(g, \gamma) \in
\Sigma^{\perp}$ and $z \in D(H)$. For example, if $(g, \gamma) \in
\Sigma^{\perp}$, then
\begin{align*}
(x_{i} \ot \varepsilon) \big( (\gamma \ot g) z \big) & = (g_{i}
\ot \chi_{i}^{-1}) (\gamma \ot g) \, (x_{i} \ot \varepsilon) (z) +
(x_{i} \ot \varepsilon) (\gamma \ot g) \, (1 \ot \varepsilon)
(z)\\
& = \gamma (g_{i}) \chi_{i}^{-1} (g) \, (x_{i} \ot \varepsilon) (z) \\
& = (x_{i} \ot \varepsilon) (z).
\end{align*}
It follows that the generators of $A$ are contained in the image
of $\pi^{*}$, so $A \subseteq \textnormal{Im} (\pi^{*})$. Using Remark~\ref{dimA}
we compute
$$
\dim \textnormal{Im} (\pi^{*}) = \frac{\dim D(H)}{ \dim K} =
\frac{|\Gamma|^{2} 2^{2n}}{|\Sigma^{\perp}|} = |\Sigma| \, 2^{2n} =
\dim A
$$
and so  $A = \textnormal{Im} (\pi^{*})$.
\end{proof}

Recall  the notion of the Drinfeld double of $V$
from Section~\ref{double of quantum linear space}.
We have $D(V) = W \oplus
W^{*} \in \, ^{\Sigma}_{\Sigma} \mathcal{YD}$, where $W$ is the
quantum linear space associated to the datum
\[
\big( (g_{1},
\chi_{1}^{-1}), \dots, (g_{n}, \chi_{n}^{-1}), \varphi_{1}, \dots,
\varphi_{n} \big),\quad  \mbox{with }  \varphi_{i} \in\widehat{\Sigma},\,
\varphi_{i} (g, \chi) = \chi_{i} (g),\, i=1,\dots,n.
\]
Define a bicharacter $r_\Sigma: \Sigma\times \Sigma \to k^\times$ by
\[
r_\Sigma \left((g,\chi),\,(g',\,\chi')\right) = \chi'(g)   .
\]
The diagonal of this bicharacter is a quadratic form $q_{\Sigma}  : \Sigma \to k^{\times}$,
\[
q_{\Sigma} (g, \chi) = \chi (g),\qquad (g, \gamma) \in \Sigma. 
\]
Then $D(V)\in \Z_{sym}(\C(\Sigma,\, r_\Sigma))_-$.

\begin{theorem}
\label{Zad}
Let $(\Gamma, q, V, r)$ be a metric quadruple such that $V \in \C(\Gamma,\, q)$ is self-dual. 
There is an equivalence of braided tensor categories:
$$
\mathcal{Z} \big( \mathcal{C} (\Gamma,\, q,\, V,\, r)
\big)_{\textnormal{ad}} \cong \mathcal{C} (\Sigma,\, q_{\Sigma},\, D(V),\,
r_{D(V)}),
$$
where 
$r_{D(V)}$ is the canonical symplectic form on $D(V)$ defined in \eqref{rDV}.
\end{theorem}
\begin{proof}
Let $H = \mathfrak{B} (V) \# k[\Gamma]$ and let $A$ be the Hopf
subalgebra of $D(H)^{*}$ generated by the group-like elements
$g_{i} \ot \chi_{i}^{-1}$ and by the skew-primitive elements
$x_{i} \ot \varepsilon$ and $g_{i} \ot r' (-, x_{i})$, $i = 1,
\dots, n$, where $r'$ is an $r$-form on $H$ whose restriction to
$V \ot V$ is non-degenerate. Using
Proposition~\ref{quotient_of_D(H)}, we have
$$
\mathcal{Z} ( \Corep(H) )_{\textnormal{ad}} = \big( \Rep
(D(H)^{\textnormal{cop}}) \big)_{\textnormal{ad}} = \Rep \big( A^{*
\, \textnormal{cop}} \big) =  \Rep \big( A^{*}\big)^{\textnormal{op}} \simeq \Rep(A^{*}) = \Corep (A).
$$
where the equivalence between $\Rep(A^{*})$ and its opposite
follows from the fact that $\Rep (A^{*})$ is braided.

We claim that $A \cong \mathfrak{B} (D(V)) \# k [\Sigma]$. Indeed,
it is easy to check that for each  $(g, \gamma) \in \Sigma$ we
have
\begin{align*}
(g \ot \gamma) (x_{i} \ot \varepsilon) & = \chi_{i}^{-1} (g)
(x_{i} \ot \varepsilon) (g \ot \gamma),\\
(g \ot \gamma) (g_{i} \ot r (-, x_{i})) & = \chi_{i}^{-1} (g)
(g_{i} \ot r (-, x_{i})) (g \ot \gamma),\\
(g_{i} \ot r (-, x_{i})) (x_{j} \ot \varepsilon) & = \chi_{j}^{-1}
(g_{i}) (x_{j} \ot \varepsilon) (g_{i} \ot r (-, x_{i})),
\end{align*}
for all $i$, $j = 1, \dots, n$. Note that $W$ is self-dual
because $V$ is self-dual. Thus, there exists a basis
$\{y_{i}\}_{i=1}^{2n}$ of $D(V)$ such that $y_{i}$, $y_{n + i} \in
D(V)_{(g_{i}, \chi_{i}^{-1})}^{\varphi_{i}}$. It follows from the
above, that the map $A \to \mathfrak{B} (D(V)) \# k[\Sigma]$,
given by
$$
g \ot \gamma \mapsto (g, \gamma), \quad x_{i} \ot \varepsilon
\mapsto y_{i}, \quad g_{i} \ot r (-, x_{i}) \mapsto y_{n+ i}
$$
for all $(g, \gamma) \in \Sigma$ and $i = 1, \dots, n$, is a Hopf
algebra isomorphism.

The braiding on $\mathcal{Z} \big( \mathcal{C} (\Gamma, q, V, r)
\big)_{\textnormal{ad}}$ is obtained by restriction of the
braiding of $\Z\big( \C(\Gamma, q, V, r) \big)$. It
corresponds to the the braiding on $\Corep(A)$ coming from the
restriction to $A$ of the canonical $r$-form
$r_{D(H)^{*}} : D(H)^{*} \ot D(H)^{*} \to k$.
on $D(H)^{*}$. The latter is  given by
$$
r_{D(H)^{*}} (\alpha, \beta) = \sum_{h \in \Gamma, P \subseteq
\{1, \dots, n\}} \alpha ( \varepsilon \ot hx_{P}) \beta
((hx_{P})^{*} \ot 1), \quad \alpha, \beta \in D(H)^{*}.
$$
We have
\begin{align*}
r_{D(H)^{*}} ( g \ot \gamma,  g' \ot \gamma') & = \sum_{h, P} \varepsilon (g) \gamma (hx_{P}) (hx_{P})^{*} (g') \gamma'(1) = \gamma (g'),\\
r_{D(H)^{*}} ( x_{i} \ot \varepsilon, x_{j} \ot \varepsilon ) & = \sum_{h, P} \varepsilon (x_{i}) \varepsilon (hx_{P}) (hx_{P})^{*} (x_{j} \varepsilon (1)) = 0,\\
r_{D(H)^{*}} \big( x_{i} \ot \varepsilon, g_{j} \ot r (-, x_{j}) \big) & = \sum_{h, P} \varepsilon(x_{i}) \varepsilon (hx_{P}) (hx_{P})^{*} (g_{j}) r(1, x_{j}) = 0, \\
r_{D(H)^{*}} \big( g_{j} \ot r (-, x_{j}), x_{i} \ot \varepsilon \big) & = \sum_{h, P} \varepsilon (g_{j}) r (hx_{P}, x_{j}) (hx_{P})^{*} (x_{i}) \varepsilon (1) = r (x_{i}, x_{j}),\\
r_{D(H)^{*}} \big( g_{i} \ot r (-, x_{i}), g_{j} \ot r (-, x_{j})
\big) & = \sum_{h, P} \varepsilon (g_{i}) r (hx_{P}, x_{i})
(hx_{P})^{*} (g_{j}) r (1, x_{j}) = 0.
\end{align*}

Thus, $\Corep (A, r_{D(H)^{*}}|_{A\ot A}) \simeq \mathcal{C} (\Sigma,
q_{\Sigma}, D(V), s)$, where the quadratic form $q_{\Sigma}  : \Sigma \to
k^{\times}$ is given by  $q_{\Sigma} (g, \gamma) = \gamma (g)$ for all $(g,\,
\gamma) \in \Sigma$  and the matrix of $s : D(V) \ot D(V) \to k$
with respect to the basis $\{y_{i}\}$ is the block matrix
\begin{center}
$\left(\begin{array}{cc} 0 & 0 \\
X^{t} & 0
\end{array}\right)$.
\end{center}
Here $X^{t}$ is the transpose of the matrix $X = \big( s (x_{i},
x_{j}) \big)_{i,j}$. Changing $s$ by  a cocycle deformation
$s^{\sigma}$ will not change the braided equivalence class of
$\C(\Sigma, q_{\Sigma}, D(V), s)$. As explained in Section~\ref{proof section},
we can choose invariant $\sigma$ such that $s^{\sigma}$ is alternating.
The matrix of $s^{\sigma}$ with respect to the basis $\{y_{i}\}$ 
is then
\begin{center}
$\frac{1}{2} \left(\begin{array}{cc} 0 & -X \\
X^{t} & 0
\end{array}\right).$
\end{center}
This matrix is easily seen to be congruent to $\left(\begin{array}{cc} 0 & -I_{n} \\
I_{n} & 0 \end{array}\right)$. 
%More precisely, we have:
%\begin{center}
%$\left(\begin{array}{cc} I_{n} & 0 \\
%0 & X^{t}
%\end{array}\right) \left(\begin{array}{cc} 0 & -I_{n} \\
%I_{n} & 0
%\end{array}\right) \left(\begin{array}{cc} I_{n} & 0 \\
%0 & X
%\end{array}\right) = \left(\begin{array}{cc} 0 & -X \\
%X^{t} & 0
%\end{array}\right)$
%\end{center}
So after a change of basis the matrices of $s^{\sigma}$ and  $r_{D(V)}$ coincide. Therefore, $\Corep (A,
r_{D(H)^{*}}|_{A\ot A}) \simeq \mathcal{C} (\Sigma, q_{\Sigma}, D(V),
r_{D(V)})$.
\end{proof}

Let $G$ be a finite group.
Recall that a tensor category $\C$ is called a {\em $G$-extension} of a tensor category $\A$
if there is a faithful grading $\C =\bigoplus_{g\in G}\, \C_g$ such that that $\C_e\cong \A$.

\begin{corollary}
$\Z(\C(\Gamma,\, q,\, V,\, r))$ is a $\widehat{\Sigma^{\perp}}$-extension of
$\C(\Sigma,\, q_{\Sigma},\, D(V),\, r_{D(V)})$.
\end{corollary}

%%%%%%%%%%%%%%%%%%%%%%%%%%%%%%%%%%%%%%%%%%%%%%%%%%%%%%
%%%%%%%%%%%%%%%%%%%%%%%%%%%%%%%%%%%%%%%%%%%%%%%%%%%%%%
%%%%%%%%%%%%%%%%%%%%%%%%%%%%%%%%%%%%%%%%%%%%%%%%%%%%%%
\bibliographystyle{ams-alpha}

\begin{thebibliography}{A}

\bibitem[A14]{A14} N.~Andruskiewitsch,
\textit{On finite-dimensional Hopf algebras}, arXiv:1403.7838v1 [math.QA].

\bibitem[AS98]{AS98} N.~Andruskiewitsch, H.-J.~Schneider, \textit{Lifting of quantum linear
spaces and pointed Hopf algebras of order $p^3$}, J.\ Algebra
\textbf{209}, 658-691 (1998).

\bibitem[AS10]{AS10} N.~Andruskiewitsch, H.-J.~Schneider,
\textit{On the classification of finite-dimensional pointed Hopf algebras}, Ann. of Math. (2)
\textbf{171}, 375-417 (2010).

\bibitem[An13]{An13} I.~Angiono,
\textit{On Nichols algebras of diagonal type},
J.\ Reine Angew. Math. \textbf{683} (2013), 189-251

\bibitem[B03]{B03} M.~Beattie,
\textit{Duals of pointed Hopf algebras}, J.\ Algebra, \textbf{262}, 54--76 (2003).

\bibitem[BC06]{BC06} J.~Bichon, G.~Carnovale,
\textit{Lazy cohomology: an analogue of the Schur multiplier for
arbitrary Hopf algebras}, J. Pure Appl. Algebra \textbf{204}, no.\ 3, 627-665 (2006).

\bibitem[BN15]{BN15} C.~Bontea, D.~Nikshych,
\textit{On the Brauer-Picard group of a finite symmetric tensor category},
J.\ Algebra, \textbf{440} (2015), 187--218.

%\bibitem[CD99]{CD99} 
%S.~Caenepeel, S.~D\u{a}sc\u{a}lescu, 
%\textit{On pointed Hopf algebras of dimension $2^{n}$}, Bull. London Math. Soc. \textbf{31} (1999),
%no. 1, 17-24.

\bibitem[CC04]{CC04} G.~Carnovale, J.~Cuadra,
\textit{Cocycle twisting of $E(n)$-module
algebras and applications to the Brauer group}, K-Theory
\textbf{33} (3), 251-276 (2004).

\bibitem [Da10]{Da10} A.~Davydov,
\textit{Twisted automorphisms of Hopf algebras},
in \textit{Noncommutative structures in mathematics and physics}, 103-130,
K. Vlaam. Acad. Belgie Wet. Kunsten (KVAB), Brussels, 2010.

\bibitem [DN13]{DN13}  A.~Davydov, D.~Nikshych,
\textit{The Picard crossed module of a braided tensor category},
Algebra and Number Theory, \textbf{7} (2013), no.\ 6, 1365--1403.

\bibitem[De02]{De02} P.~Deligne, 
\textit{Cat\'egories tensorielles},  
Moscow Math.\ J.\  \textbf{2} , no.\ 2, 227--248 (2002).

\bibitem[DGNO10]{DGNO10}
V.~Drinfeld, S.~Gelaki, D.~Nikshych, V.~Ostrik,
\textit{On braided fusion categories I}, Selecta Mathematica, \textbf{16} (2010), no. 1,  1--119.

\bibitem[EGNO15]{EGNO15}
P.~Etingof, S.~Gelaki, D.~Nikshych, V.~Ostrik, {\em  Tensor categories},
Mathematical Surveys and Monographs,
\textbf{205}, American Mathematical Society (2015).

\bibitem[GN08]{GN08} S.~Gelaki and D.~Nikshych,
\textit{Nilpotent fusion categories},
\textit{Advances in Mathematics} \textbf{217} (2008), no. 3, 1053--1071.

%\bibitem[J08]{J08} M. Janjic, \textit{A proof of generalized Laplace's Expansion
%Theorem}, Bull. Soc. Math. Banja Luka \textbf{15} (2008), 5-7.

\bibitem[JS93]{JS93} A. Joyal, R. Street, 
\textit{Braided tensor categories}, Adv.\ Math., \textbf{102},  20--78 (1993).


\bibitem[Ma01]{Ma01} A.~Masuoka,
\textit{Defending the negated Kaplansky conjecture}, Proc. Amer. Math. Soc. \textbf{129} (2001), 3185--3192.

%\bibitem[Ma08]{Ma08} A.~Masuoka,
%\textit{Abelian and non-abelian second cohomologies of quantized
%enveloping algebras}, J.\ Algebra \textbf{320} (2008), no. 1, 1--47.

\bibitem[Mo11]{Mo11} M.~Mombelli, \textit{Representations of tensor categories coming from quantum linear spaces},
J.\ London Math.\ Soc.\ (2) \textbf{83} (2011) 19--35.

\bibitem[M93]{M93} S.~Montgomery, \textit{Hopf Algebras and Their Actions on Rings}, CBMS Regional
Conference Series in Mathematics, \textbf{82}, AMS, (1993).

\bibitem[Na06]{Na06} S.~Natale, \textit{$R$-matrices and Hopf Algebra
Quotients}, Int. Math. Res. Notices \textbf{2006} (2006), 1--18.

\bibitem[Ne04]{Ne04} A.~Nenciu,
\textit{Quasitriangular pointed Hopf algebras constructed by Ore extensions},
Algebra and Representation Theory \textbf{7}, 159--172 (2004).

\bibitem[Ni78]{Ni78} W.D.~Nichols, \textit{Bialgebras of type one},
Comm.\ Algebra \textbf{6} (15), 1521--1552 (1978).

\bibitem[PvO99]{PvO99} F.~Panaite, F.~van Oystaeyen, \textit{Quasitriangular structures for
some pointed Hopf algebras of dimension $2^{n}$}, Comm.\ Algebra
\textbf{27} (10), 4929--4942 (1999).

\bibitem[R93]{R93} D.E.~Radford, \textit{Minimal quasitriangular Hopf algebras}, J.
Algebra \textbf{157} (1993), no. 2, 285--315.

\bibitem[R94]{R94} D.E.~Radford, \textit{On Kauffman's knot invariants arising
from finite dimensional Hopf algebras}, in ``Advances in Hopf
algebras", pp. 205--266, J.Bergen and S.Montgomery (eds.), Lecture
Notes in Pure and Appl.Math.\ \textbf{158}, Marcel Dekker, New
York, 1994.

\bibitem[R11]{R11} D.E. Radford.
\textit{Hopf algebras}, \textit{World Scientific} (2011).




\end{thebibliography}

\end{document}